\documentclass[11pt]{article}
\usepackage{amsmath,amsthm,amsfonts,epsfig}
\usepackage[dvipsnames]{color}
\usepackage{graphicx}
\usepackage{amssymb}
\usepackage{enumitem}
\usepackage[colorlinks=true,linkcolor=blue,citecolor=blue]{hyperref}
\usepackage[mathscr]{euscript}
\usepackage{tikz}
\usepackage{changes}
\usepackage{xcolor}

\newtheorem{thm}{Theorem}[section]
\newtheorem{lemma}[thm]{Lemma}
\newtheorem{cor}[thm]{Corollary}
\theoremstyle{definition}
\newtheorem{defn}[thm]{Definition}

\newtheorem{remark}[thm]{Remark}
\newtheorem{prop}[thm]{Proposition}

\newcommand{\N}{\mathbb{N}}
\newcommand{\Z}{\mathbb{Z}}
\newcommand{\Q}{\mathbb{Q}}
\newcommand{\R}{\mathbb{R}}

\newcommand{\eps}{\varepsilon}
\newcommand{\pip}{\varphi}

\newcommand{\rad}{\text{rad}}
\newcommand{\Ead}{\mathcal{E}_a^Z}
\newcommand{\Eadp}{\mathcal{E}_a^{Z_+}}
\newcommand{\Pa}{\mathcal{P}_a}

\newcommand{\bbR}{\mathbb{R}}
\newcommand{\bbN}{\mathbb{N}}

\definechangesauthor[name=Gianmarco, color=blue]{G}

\def\vint_#1{\mathchoice
          {\mathop{\vrule width 6pt height 3 pt depth -2.5pt
                  \kern -8pt \intop}\nolimits_{#1}}%
          {\mathop{\vrule width 5pt height 3 pt depth -2.6pt
                  \kern -6pt \intop}\nolimits_{#1}}%
          {\mathop{\vrule width 5pt height 3 pt depth -2.6pt
                  \kern -6pt \intop}\nolimits_{#1}}%
          {\mathop{\vrule width 5pt height 3 pt depth -2.6pt
                  \kern -6pt \intop}\nolimits_{#1}}}

\DeclareMathOperator{\Mod}{\text{Mod}}

\DeclareMathOperator{\Lip}{\rm Lip}

\newcommand{\distz}{\text{dist}_Z}
\def\XXint#1#2#3{{\setbox0=\hbox{$#1{#2#3}{\int}$ }
\vcenter{\hbox{$#2#3$ }}\kern-.6\wd0}}

\setlist[enumerate]{label=(\arabic*)} %{label=\emph{(\arabic*})}

\numberwithin{equation}{section}

\newcommand{\co}{\colon}

\title{Regularity of Solutions to the Fractional Cheeger-Laplacian on Domains in Metric Spaces of Bounded Geometry}

\author{Sylvester Eriksson-Bique, Gianmarco Giovannardi, Riikka Korte, \\
Nageswari Shanmugalingam, and Gareth Speight}

\begin{document}
\maketitle

\begin{abstract}
We study existence, uniqueness, and regularity properties of the Dirichlet problem related to fractional Dirichlet energy
minimizers in a complete doubling metric measure space $(X,d_X,\mu_X)$ 
satisfying a $2$-Poincar\'e inequality. 
Given a bounded domain $\Omega\subset X$ with $\mu_X(X\setminus\Omega)>0$, and a function $f$ in the Besov class
$B^\theta_{2,2}(X)\cap L^2(X)$, we study the problem of finding
a function $u\in B^\theta_{2,2}(X)$ such that $u=f$ in $X\setminus\Omega$ and 
$\mathcal{E}_\theta(u,u)\le \mathcal{E}_\theta(h,h)$ whenever $h\in B^\theta_{2,2}(X)$ with $h=f$ in $X\setminus\Omega$.
We show that such a solution always exists and that this solution is unique.
We also show that the solution is locally H\"older continuous on $\Omega$, and satisfies a non-local
maximum and strong maximum principle. Part of the results in this paper extend the 
work of Caffarelli and Silvestre in the Euclidean setting and Franchi and Ferrari in Carnot groups.
\end{abstract}



\tableofcontents

\section{Introduction}\label{Sect:Intro}

The development of analysis on metric measure spaces in recent decades has provided a fruitful study of upper gradient
$p$-energy minimizers in complete metric measure spaces equipped with a doubling measure supporting a $p$-Poincar\'e inequality.
Here the notion of upper gradient is the metric space generalization of the norm of the derivative from the seminal work of Heinonen and Koskela~\cite{HeiK}.
An application of the subsequent work of Cheeger~\cite{Ch} gave a differential structure on such  a metric measure space
with respect to which every Lipschitz function enjoys a first order Taylor approximation property, and the differential structure
can be equipped with a measurable inner product so that the induced norm on the differential of a Lipschitz function is comparable
to the minimal $p$-weak upper gradient. Therefore, one can, instead of minimizing the upper gradient energy, minimize the
energy given by integrating the $p$-th power of the norm of the differential. Such energy minimizers are 
upper gradient $p$-energy
quasiminimizers in the sense of~\cite{KiSh}, and hence have regularity properties such as local H\"older continuity and the Harnack
inequality that the upper gradient $p$-energy minimizers also satisfy. 

In considering the Cheeger differential structure $D_X$ on the metric measure space $(X,d_X,\mu_X)$, 
thanks to the inner product on this structure we have an induced Dirichlet form (corresponding to $p=2$)
in the sense of~\cite{FOT}. Thus the theory of Dirichlet forms yields a Cheeger Laplacian operator $\Delta_X$.
The fractional Laplace operator $(-\Delta_X)^{\theta}$, defined via spectral theory by using $\Delta_X$ and the 
associated Dirichlet forms, is a
non-local operator on $X$. The goal of this paper is to study existence and regularity properties of the solution 
to the non-local equation $(-\Delta_X)^\theta u=0$ on a bounded domain $\Omega\subset X$ with 
Dirichlet data $u=f$ on $X\setminus\Omega$ for $f$ in the suitable function class on $X$. The suitable function class here
is the inhomogeneous Besov class $B^\theta_{2,2}(X)\cap L^2(X)$. The permissible range of $\theta$ is $0<\theta<1$.

To achieve the goals described above, we use the line of investigation implemented by Caffarelli and Silvestre in~\cite{CS}.
Additional tools and structures from~\cite{BGMN, BLS, GKS,G, KiSh} are also key components in our proofs. The following are the
main results of this paper. The first theorem below establishes the existence of a solution to the fractional Laplacian problem with
given Dirichlet data. Here, with $\Delta_X$ a choice of 
the Cheeger Laplacian on $X$, and $0<\theta<1$, we set 
\begin{equation}\label{eq:frac-Dirich-1}
\mathcal{E}_{\theta}(f,f):=\int_X ((-\Delta_X)^{\theta/2} f)^2\, d\mu_X.
\end{equation}

\begin{thm}\label{thm:minimization} 
Let $f \in B^\theta_{2,2}(X)\cap L^2(X)$, and $\Omega$ be a bounded domain in $X$
with $\mu_X(X\setminus\Omega)>0$.
Then there is a unique $u \in B^\theta_{2,2}(X)$ 
with $u=f$ in $X\setminus\Omega$
such that whenever $h\in B^\theta_{2,2}(X)$ with $h=f$ in $X\setminus\Omega$, we have
\begin{equation}\label{eq:fract-Laplace}
\mathcal{E}_{\theta}(u,u)\leq \mathcal{E}_{\theta}(h,h).
\end{equation}
Equivalently, we have 
\[
\mathcal{E}_\theta(u,h)=0
\]
whenever $h\in B^\theta_{2,2}(X)$ such that
$h$ has compact support in $\Omega$.
\end{thm}

In the Euclidean setting the existence and uniqueness results for the fractional Dirichlet problem  were obtained 
in~\cite{FKV, HJ,RO}.
The notion related to~\eqref{eq:fract-Laplace} is given in the next section, see Definition~\ref{def:fract-sol} below. 
The proof of the above theorem also shows that if $f$ is in $B^\theta_{2,2}(X)$ but not necessarily in 
$L^2(X)$, then the solution still exists provided we can make sense of $\mathcal{E}_\theta(f,f)$. Indeed,
if $f$ is a non-zero constant, or a perturbation of a nonzero constant by a 
function in $B^\theta_{2,2}(X)\cap L^2(X)$, then $\mathcal{E}_\theta(f,f)$ should make sense.

The next theorem discusses the regularity properties of the solution. 

\begin{thm}\label{thm:mainthm1} 
Let $0<\theta<1$.
Suppose $(X,d_X,\mu_X)$ is a complete and doubling metric measure space that satisfies a
$2$-Poincar\'e inequality, and that $\Omega\subset X$ is a bounded domain with $\mu_X(X\setminus\Omega)>0$. 
Suppose further that $f\in B^\theta_{2,2}(X)\cap L^2(X)$ is a solution to~\eqref{eq:fract-Laplace}.
Then $f$ is locally H\"older continuous on $\Omega$. Moreover, if $f\ge 0$ on $X$, then $u$ satisfies a 
Harnack inequality on balls $B\subset X$ for which $2B\subset\Omega$.
\end{thm}

The pioneering work related to this problem in the setting of Euclidean domains is due to Cafarelli and Silvestre \cite{CS}.
They proved a Harnack inequality for functions $u:\R^n \to [0,\infty)$ which satisfy
\[
(-\Delta)^\theta u(x)= 0, \qquad x \in \Omega
\]
for a given Euclidean domain 
$\Omega\subset\R^n$. To do this they first consider extensions of Besov functions on $\R^n$ to Sobolev functions
on a suitably weighted $\R^n\times\R$, solve a corresponding Dirichlet problem for the weighted analogue of the
standard Laplacian on $\R^n\times(0,\infty)$ with boundary
data the suitable Besov function on $\R^n$, and then study the boundary behavior on $\partial(\R^n\times(0,\infty))$) 
of such a solution when the boundary datum itself is a Besov energy minimizer.
They show that in this case, the solution on $\Omega_0:=\R^n\times(0,\infty)$
has an extension to all of $\Omega\times\R$ that is (weighted) harmonic on $\Omega\times\R$. Then the 
knowledge that the harmonic functions are locally H\"older continuous and satisfy a Harnack inequality can be 
used to verify the corresponding property for the Besov energy minimizer on $\Omega$. This approach was 
extended in \cite{FF} to Carnot groups and in~\cite{BGMN} to the parabolic setting. 
For a related non-local problem in the manifold setting, see~\cite{CSYG}.
We follow the prescription of~\cite{CS} and consider
the metric space $Z=X\times\R$, equipped with the metric
\[
d_Z((x_1,y_1),(x_2,y_2)):=\sqrt{d_X(x_1,x_2)^2+(y_1-y_2)^2},
\]
and, for $a=1-2\theta$, the measure $\mu_a$ given by $d\mu_a(x,y)=|y|^a\, dy\, d\mu_X(x)$. 
The next section describes the setting of this paper in greater detail.

The solution obtained in the proof of Theorem~\ref{thm:mainthm1} was via an 
extension of the function $f\in B^\theta_{2,2}(X)\cap L^2(X)$ from
$X\times\{0\}$ to $Z_+:=X\times(0,\infty)$ using a modified heat extension given in~\cite{BLS}. 
The final main theorem of this paper is that one can achieve this extension also by solving the Dirichlet problem
on $U_\Omega:=Z_+\cup Z_-\cup (\Omega\times\{0\})$ where $Z_-=X\times(-\infty,0)$.

\begin{thm}\label{thm:main3}
With the hypotheses given in Theorem~\ref{thm:mainthm1}, a function $f\in B^\theta_{2,2}(X)\cap L^2(X)$ is a solution
to~\eqref{eq:fract-Laplace} if and only if $f$ is the trace on $X\times\{0\}=\partial Z_+$ of the solution,
from the homogeneous Newton-Sobolev class $D^{1,2}(Z)$, to the Dirichlet
problem related to the equation $\Delta_a u=0$ on $U_\Omega$ with boundary data $f$. Moreover, such a solution
is unique in that if $h$ is another solution from $D^{1,2}(Z)$ with $h=f$ in $\partial U_\Omega$, then $h=f$ in $X$. 
Furthermore, a maximum principle and a strong maximum principle hold:
\[
\text{esssup}_{x\in\Omega}f(x)\le \text{esssup}_{w\in X\setminus\Omega}f(w),
\]
and if there is $x_0\in\Omega$ such that $\text{esssup}_{x\in X}f(x)=f(x_0)$, then $f$ is constant on $X$. 
\end{thm}

Here, by referring to $h\in D^{1,2}(Z)$ with $h=f$ in $\partial U_\Omega$ we mean that the trace of $h$ on $\partial U_\Omega$
is $\mu_X$-almost everywhere equal to $f$. To make sense of this, we do develop the notion of trace in the setting here,
see Section~\ref{Sect:Traces} below. Moreover, when we say that $\text{esssup}_{x\in X}f(x)=f(x_0)$ for some $x_0\in \Omega$,
we consider $f$ to be the continuous representative in $\Omega$ obtained from Theorem~\ref{thm:mainthm1}, with the
understanding that in $X\setminus\Omega$ the function $f$ is well-defined only $\mu_X$-almost everywhere.

The non-local nature of the fractional Laplacian is reflected in the non-local nature of the maximum and strong maximum 
principle. In the Euclidean setting, the maximum principle was obtained
in~\cite{CaSi,CLL, RO}.

The work~\cite{BGMN} studied scale-invariant Harnack inequalities for fractional powers of smooth parabolic and elliptic operators
on Euclidean spaces, extending the result of~\cite{CS} to this generality. Indeed, the smoothness assumption seems to be cosmetic
there, and it is not difficult to see that the work of~\cite{BGMN} extends also to the setting of sub-Riemannian manifolds.
It was pointed out in~\cite{BGMN} that their methods
extend to a general class of Dirichlet forms and associated infinitessimal generator as the elliptic operator. 

During the preparation of this manuscript, we became aware of the concurrent work by Baudoin, Lang, and Sire~\cite{BLS}, which 
established the Harnack principle of solutions to fractional Laplace problems in the context of strongly local Dirichlet forms
that satisfy a $2$-Poincar\'e inequality, and further studied an analog of the boundary Harnack principle for 
the case that $\Omega$ is an inner uniform domain in $X$.
Their approach, as well as that of~\cite{BGMN}, is based on spectral theory and 
gave us valuable tools to use in the study undertaken here. In our setting,
we consider a specific Dirichlet form given by the measurable inner product structure on a choice of
Cheeger differential structure available on the metric space. Given the quasiconvexity of $X$ (a consequence of the measure
being doubling and supporting a Poincar\'e inequality), the Dirichlet form of interest here satisfies the hypotheses of~\cite{BLS}.
Hence the Harnack inequality of the above theorem follows directly from~\cite{BLS}. In
Theorem~\ref{thm:mainthm1} we combine the tools developed
in~\cite{BLS} with additional tools related to potential theory in metric setting to add to the results of~\cite{BLS} in our 
context. We also connect the domain of the fractional Laplacians
explicitly to Besov spaces, traces and the upper gradient approach. In particular, we 
use their results related to the explicit Poisson-type extension~\eqref{eq:explicit}, see also~\cite{BGMN}.

There is a rich collection of mathematical literature on fractional orders of operators in smooth setting, and it is not possible
to list them all here. We direct the interested reader to the references cited above as well as the papers
cited in them.

\subsection*{Acknowledgements}
S.E-B. was partially supported by the National Science Foundation (U.S.) grant No. DMS-1704215 
and by the Finnish Academy under Research Postdoctoral Grant No. 330048. R.K.~was supported 
by Academy of Finland Grant No. 308063. N.S.~was partially supported by the National Science 
Foundation (U.S.) grant No. DMS~\#1800161. G.S. was supported by Simons Collaboration Grant No. 576219. G.G.  was supported by Horizon 2020~\# 777822: GHAIA 
and by  MEC-Feder grant MTM2017-84851-C2-1-P.

The authors are thankful to IMPAN for hosting the semester ``Geometry and analysis in 
function and mapping theory on Euclidean and metric measure space'', where part of this 
research was conducted. This work was also partially supported by the grant \#346300 for 
IMPAN from the Simons Foundation and the matching 2015-2019 Polish MNiSW fund.

The authors thank Yannick Sire for helpful discussions and for sharing an early manuscript of~\cite{BLS} with us.

\section{Background and Notation}\label{Sect:Background}

\subsection{Newton-Sobolev spaces and related notions}

In this paper we are concerned with a metric measure space $(X,d_X,\mu_X)$.
We first start with the notion of $2$-modulus of a family of curves in $X$. Given a family $\Gamma$ of curves in $X$,
the $2$-modulus of this family is the number
\[
\Mod_2(\Gamma):=\inf_\rho \int_X\rho^2\, d\mu_X,
\]
where the infimum is over all non-negative Borel measurable functions $\rho$ on $X$ that satisfy $\int_\gamma\rho\, ds\ge 1$
for each locally rectifiable curve $\gamma\in\Gamma$.

The notion of upper gradients from~\cite{HeiK} forms the foundation of first order analysis in metric measure spaces.
Given a metric space  $(X,d_X)$, 
a non-negative Borel function $g$ on $X$ is an upper gradient of a map $u\colon X\to \mathbb{R}\cup\{-\infty,\infty\}$ if
\[
|u(\gamma(b))-u(\gamma(a))|\leq \int_{\gamma}g \, ds
\]
for every rectifiable curve $\gamma\colon [a,b]\to X$. The right-hand side of the above is required to be infinite when
at least one of $u(\gamma(b))$, $u(\gamma(a))$ is not finite. We say that $g$ is a $2$-weak upper gradient (or weak upper 
gradient for short) if there is a family $\Gamma$ of curves in $X$ such that $(u,g)$ satisfies the above inequality for each
non-constant compact rectifiable curve in $X$ that does not belong to $\Gamma$ and $\Mod_2(\Gamma)=0$. 

In this paper, we extend the study of potential theory associated with the fractional Laplacian 
to a complete doubling metric measure space $(X,d_X,\mu_X)$ supporting a $2$-Poincar\'e inequality. 

\begin{defn}\label{deff:SobolevN-D}
The Newton-Sobolev space $N^{1,2}(X)$
 of all functions $f:X\to\R$ with the property that $\int_X |f|^2\, d\mu_X<\infty$
and with $\inf_g\int_Xg^2\, d\mu_X < \infty$, where the infimum is over all $2$-weak upper gradients $g$ of $f$,
is a Banach space (see~\cite{ HKSTbook,N}). 

Following~\cite{HKSTbook}, by 
$D^{1,2}(X)$ we mean the class of functions $f\in L^1_{loc}(X)$ with an upper gradient $g\in L^2(X)$. 
\end{defn}

The space $N^{1,2}(X)$ is also called the \emph{inhomogeneous}
Newton-Sobolev space, while the space $D^{1,2}(X)$ is also called the \emph{homogeneous} Newton-Sobolev space as
it will contain nonzero constant functions as well.

Just as sets of measure zero play a role in the study of $L^p$-spaces, the sets of $2$-capacity zero play a role in the study
of Sobolev spaces. 

\begin{defn}\label{deff:2-Cap}
Given a set $E\subset X$, the $2$-capacity of the set is the number
\[
\text{Cap}_2(E):=\inf_{(u,g)} \int_X|u|^2\, d\mu_X+\int_Xg^2\, d\mu_X,
\]
where the infimum is over all pairs of functions $(u,g)$ with $u\in N^{1,2}(X)$ satisfying $u\ge 1$ on $E$ and 
$g$ a $2$-weak upper gradient of $u$.
\end{defn}

We assume in this paper that the measure $\mu_X$ is doubling, namely there is a constant $C\ge 1$ such 
that $\mu_X(B(x,2r))\le C\, \mu_X(B(x,r))$ for all $x\in X$ and $r>0$.
If $X$ is connected and $\mu_X$ is doubling, there exist constants $c,C>0$ and $b_l,b_u>0$ for which
\begin{equation} \label{eq:volumegrowth}
c\left(\frac{r}{R}\right)^{b_l}\leq \frac{\mu_X(B(x,r))}{\mu_X(B(x,R))} \leq  C\left(\frac{r}{R}\right)^{b_u}
\end{equation}
for each $0<r<R<\infty$. If $b_{u}=b_{l}$ then the space is Ahlfors $b_u$-regular.

The assumption of connectedness of $X$ is not a loss of generality. Indeed, the assumption that $X$ supports a 
$2$-Poincar\'e inequality (see the next paragraph) immediately implies that $X$ is connected.

We also assume that $(X,d_X,\mu_X)$ supports a $2$-Poincar\'e 
inequality, that is, there are constants $C>0$ and $\lambda\ge 1$ such that whenever $B(x,r)$ is a ball in $X$ and
$f\in N^{1,2}(X)$ and $g$ is an upper gradient of $f$, we have
\begin{equation}\label{eq:def-PI}
 \vint_{B(x,r)}|f-f_{B(x,r)}|\, d\mu_X\le C\, r\, \left(\vint_{B(x,\lambda r)}g^2\, d\mu_X\right)^{1/2}.
\end{equation}

When the doubling space $X$ supports a $2$-Poincar\'e inequality, it also supports an a priori stronger $(2,2)$-Poincar\'e 
inequality:
\[
 \vint_{B(x,r)}|f-f_{B(x,r)}|^2\, d\mu_X\le C\, r^2\, \vint_{B(x,\lambda r)}g^2\, d\mu_X.
 \]
See \cite{HK} for more information.

\begin{remark}\label{remk:Cartesian-PI}
Recall that we are interested in the weighted measure $|y|^a\,dy$ with $-1<a<1$.
By~\cite[page~10]{HKM}  the measure $|y|^a dy$ is an $A_2$-weight on $\R$, 
and hence we know that both $\R$ and $(0,\infty)$, 
equipped with the Euclidean metric and the measure $|y|^a\, dy$, supports a $2$-Poincar\'e 
inequality and this weighted measure is doubling. Hence the Cartesian product $Z=X\times\R$
as well as the Cartesian product $Z_+=X\times(0,\infty)$, equipped with 
the metric $d_Z$ and the
measure $\mu_a$, also supports a $2$-Poincar\'e inequality with $\mu_a$ doubling, see~\cite[Remark~4]{BB2}
and~\cite{BjSh} (where we use the fact that $Z_+$ is a uniform domain, see Proposition~\ref{prop:unif+PI} below). 
\end{remark}

There is a strengthening of~\eqref{eq:def-PI} under the assumption that $\{f=0\}\cap B(x,r)$ is large:
\begin{equation}\label{eq:Mazya}
\int_{B(x,r)}|f|^2\, d\mu_X\le C\, (r^2+1)\, \frac{\mu_X(B(x,r))}{\text{Cap}_2(B(x,r)\cap N_f)}\, \int_{B(x,\lambda r)}g^2\, d\mu_X,
\end{equation}
where $N_f:=\{w\in X\, :\, f(w)=0\}$. This inequality is known as the Maz'ya capacitary inequality, see~\cite{Maz}.
For the setting of doubling metric measure spaces supporting a $(2,2)$-Poincar\'e inequality, a good
reference is~\cite[Theorem~5.53]{BB}.

\subsection{Cheeger Differential Structure}

In this subsection we describe the Cheeger differential structure.

\begin{defn}\label{deff:Ch-Diff-Struct}
A \emph{system of Lipschitz charts} $\{(U_i, \varphi_i) : i \in \N\}$ for a metric measure space $(X,d_X,\mu_X)$ is a 
collection of countably many measurable sets $U_i \subset X$ and Lipschitz maps $\varphi_i \co X \to \R^{n_i}$ for 
each $i \in \N$, so that
\[
\mu_X\left(X \setminus \bigcup_{i \in \N} U_i\right) = 0
\]
and for any Lipschitz function $f \co X \to \R$ and for each $i$ the following holds. 
For almost every $x\in U_{i}$, there exists a unique $D_Xf(x)\in \bbR^{n_{i}}$ such that 
\begin{equation}\label{eq:diff}
\lim_{y\to x} \frac{f(y)-f(x)-D_Xf(x)\cdot (\varphi_i(y)-\varphi_i(x))}{d(x,y)}=0.
\end{equation}
A space with a system of Lipschitz charts is said to admit a differentiable structure and is called a 
\emph{Lipschitz differentiability space}.
\end{defn}

The equation~\eqref{eq:diff} gives the first order Taylor expansion of $f$ near $x$ with respect to the basis $\varphi_i$.
This should not be confused with the notion of weak, or distributional, derivative that is usually considered with Sobolev
spaces in the Euclidean setting.

Recall that any doubling metric measure space admitting a $(1,2)$-Poincar\'e inequality admits a differential structure by \cite{Ch},
with each $n_i\le N$ for some positive integer $N$ that depends solely on the doubling constant of the measure $\mu_X$.
By embedding each $\R^{n_i}$ into $\R^N$ if necessary, we may therefore assume that each $n_i=N$. In this case,
the uniqueness of $D_X$ is preserved by ensuring that the entries in the vector $D_X f(x)$ corresponding to the components
$n_i+1,\cdots, N$ are all zero when $n_i<N$. Hence, in our setting of $(X,d_X,\mu_X)$, we have a linear map
\[
D_X: N^{1,2}(X)\to L^2(X)^N
\] 
for some fixed positive integer $N$ that is determined by the doubling property of the
measure $\mu_X$. It was also shown in~\cite{Ch} that there is a measurable inner product structure 
related to the differential structure $D_X$, that is, for Lipschitz functions $f,h$ (and then by extension, to functions $f,h\in N^{1,2}(X)$),
for $\mu_X$-a.e.~$x\in X$ we have $\langle D_Xf(x),D_Xh(x)\rangle_x$ such that as a function of $x$ this is measurable, and there is
a constant $C>0$ that is independent of $f$ such that for $\mu_X$-a.e.~$x\in X$,
\[
\frac{1}{C}g_f(x)^2 \le \langle D_Xf(x),D_Xf(x)\rangle_x \le  C g_f(x)^2.
\]

\subsection{Besov Classes, Cheeger Harmonicity, and Dirichlet Forms}
  
In our context the replacement for the 
standard Laplacian $\Delta_X$ is the infinitesimal generator $\Delta_X$ associated with the Dirichlet form
\[
\mathcal{E}_X(f,g) = \int_X \langle D_Xf(x), D_Xg(x) \rangle_x \, d\mu_X(x)
\]
as described in~\cite{FOT}.
While $N^{1,2}(X)$ need not be a Hilbert space under the norm
$\Vert f\Vert_{L^2(X)}+\inf_g\Vert g\Vert_{L^2(X)}$ where the infimum is over all upper gradients $g$ of $f$, 
it does turn into a Hilbert space under the norm
\[
\Vert f\Vert_{L^2(X)}+\sqrt{\int_X\langle D_Xf(x),D_Xf(x)\rangle_x\, d\mu_X(x)},
\]
as seen from~\cite{Ch} or~\cite[Theorem~10]{FHK}. 

From Remark~\ref{remk:Cartesian-PI} above, we know that $Z$ also comes equipped with a choice of a 
Cheeger differential structure.
We are interested in considering a particular Cheeger structure on $Z$. This structure is obtained as a Cartesian
tensorization of the (choice of) Cheeger differential structure $D_X$ on $X$ and the standard Euclidean differential
structure on $\R$, as explained in 
Subsection~\ref{Subsect:diff-struct-tensor} below, see Theorem~\ref{diffproduct}.

\begin{defn}\label{defn:Dirich-a}
When considering the above-mentioned Dirichlet form associated with the metric space $Z$ as described at 
the end of Section~\ref{Sect:Intro} above,  the Dirichlet form obtained on $(Z,d_Z,\mu_a)$ from 
the Cartesian tensor product of the Dirichlet form $\mathcal{E}_X$ and the natural Dirichlet form on $(\R, d_{Euc},|y|^a\, dy)$
is denoted by $\Ead$.
\end{defn}

It was shown in~\cite{GKS} that
the interpolation of $L^2(X)$ with $N^{1,2}(X)$ yields Besov classes $B^\theta_{2,2}(X)\cap L^2(X)$  of functions
$f\in L^2(X)$ for which the (non-local) energy semi-norm $\Vert f\Vert_{B^\theta_{2,2}(X)}$ given by 
\begin{equation}\label{eq:BesovEnergy}
\Vert f\Vert_{B^\theta_{2,2}(X)}^2:=\int_X\int_X\frac{|f(z)-f(w)|^2}{d(z,w)^{\theta}\mu_X(B(z,d(z,w)))}\, d\mu_X(z)\, d\mu_X(w)
\end{equation}
is finite. This energy is comparable to the one given by $\mathcal{E}_\theta$ in Equation \eqref{eq:frac-Dirich-1}. 
The comparability was proved in \cite[Corollary 5.5]{G} in the Ahlfors regular case. 
However, that variants of the Gaussian bounds 
used in~\cite{G} apply also in the doubling case (see e.g.~\cite{S-C}). This leads to the comparability of the 
$B^\theta_{2,2}$-- and $\mathcal{E}_\theta$-- energies on doubling spaces satisfying a $2$-Poincar\'e inequality. 
We give a direct proof of this for the readers convenience in Proposition \ref{prop:besovenergy}. 

Yet another advantage of considering the Besov space is its identity as the trace space of a Newton-Sobolev space, namely,
$B^\theta_{2,2}(X)$ is the trace space of the homogeneous Newton-Sobolev space
$D^{1,2}(Z_+)$ consisting of all functions $f\in L^2_{loc}(Z,\mu_a)$ such that $f$ has an upper gradient $g$ in $Z_+$ 
with $\int_{Z_+}g^2\, d\mu_a$ finite. Here, $a$ and $\theta$ are related by the equation $a=1-2\theta$, see~\cite{LS}
for related results on traces of weighted Sobolev spaces and Besov spaces. For $Z_+$ as a domain in the weighted space 
$(Z,d_Z,\mu_a)$ this trace result is established in Proposition~\ref{prop-Trace} below. 
By trace class we mean that every function $u\in D^{1,2}(Z_+)$ has a trace $Tu:X\to\R$ given by
\begin{equation}\label{eq:Trace}
Tu(x)=\lim_{r\to 0^+}\vint_{B((x,0),r)\cap Z_+}u\, d\mu_a
\end{equation}
for $\mu_X$-a.e.~$x\in X$. We show that $Tu\in B^\theta_{2,2}(X)$ and that for each $f\in B^\theta_{2,2}(X)$ there exists
a function $Ef\in D^{1,2}(Z_+)$ such that $TEf=f$.
As a consequence,  we know that
whenever  $f\in B^\theta_{2,2}(X)$, there is at least one function in $D^{1,2}(Z_+)$ whose trace is $f$, and we use this to
establish the existence of Cheeger harmonic functions in $Z_+$ with trace $f$, see the discussion in
Section~\ref{Sect:Exist-CheegerHarm} below.

\begin{defn}\label{def:CheegerHarm}
Given a domain $U\subset Z$, we say that a function $u$ on $U$ is Cheeger harmonic (or Cheeger $2$-harmonic) in $U$ if
$u\in N^{1,2}_{loc}(U,\mu_a)$ and whenever $v\in N^{1,2}(U,\mu_a)$ with compact support contained in $U$ we have
$\Ead(u,v)=0$.
\end{defn}

\begin{defn}\label{def:fract-sol}
Fix a bounded domain $\Omega\subset X$ such that $\mu_X(X\setminus\Omega)>0$ and a function
$f\in B^\theta_{2,2}(X) \cap L^2(X)$. Then we call a function $u\in B^\theta_{2,2}(X)$ 
a \emph{solution to the Dirichlet
problem $(-\Delta_X)^\theta u=0$ on $\Omega$
with boundary data $f$} if $u$ satisfies~\eqref{eq:fract-Laplace}, that is,
\begin{enumerate}
\item $u=f$ almost everywhere on $X\setminus\Omega$, and
\item for all $h\in B^\theta_{2,2}(X)$ satisfying $h=f$ almost everywhere in $X\setminus\Omega$, we have
\[
\mathcal{E}_\theta(u,u)\leq \mathcal{E}_\theta(h,h).
\]
\end{enumerate}
A direct argument using calculus of variations gives the weak formulation of the Euler-Lagrange equation associated with 
the above minimization property; namely, $u$ is a minimizer in the sense of~(2) above if and only if 
for each $h \in B^\theta_{2,2}(X)$ with compact support in $\Omega$,
\[
\mathcal{E}_\theta(u,h)=0.
\]
\end{defn}

Observe that if $u\in B^\theta_{2,2}(X)$ with
$u=f$ on $X\setminus\Omega$ and $\Omega$ is bounded, then $u\in L^2(X)$ whenever $f\in L^2(X)$.

The non-local nature of the energy $\mathcal{E}_\theta$ means that we cannot replace the
global energy $\mathcal{E}_\theta$ with a local energy $\mathcal{E}^\Omega_\theta$ adapted to 
the smaller set that is $\Omega$.  
Hence the classical approaches of De Giorgi and Nash-Moser (see~\cite{BM, KiSh}) are not
applicable in the study of these solutions. We instead adapt the method set out in the Euclidean setting by
Caffarelli and Silvestre~\cite{CS}.

In addition to the construction of Cheeger harmonic functions as in Section~\ref{Sect:Exist-CheegerHarm},
we will also use the explicit extension of $f\in B^\theta_{2,2}(X)\cap L^2(X)$ to $Z_+$ 
given in~\cite[Lemma~3.1 and bottom of page~8]{BLS}. 
We show that their extension is Cheeger harmonic, and by 
uniqueness of Cheeger harmonic extensions, we obtain that both solutions coincide.
To do so, we show that the trace of their extension in the sense of~\eqref{eq:Trace} above
coincides with $f$ and that their extension also belongs to $D^{1,2}(Z_+)$.

Then, we show that the reflection of $u$ along $\partial Z_+$ given by
\[
u^*(x,y):=\begin{cases} u(x,y) &\text{ if }y\ge 0,\\
    u(x,-y) &\text{ if }y<0\end{cases}
\]
gives a function that is Cheeger harmonic in $\Omega\times\R$ and hence is a quasiminimizer in the sense of~\cite{KiSh}. 
The Cheeger harmonicity follows from uniqueness and considering a Dirichlet problem on 
$Z \setminus (X \setminus \Omega \times \{0\})$  with boundary data $f$ and observing symmetry.
Finally, application of the regularity results from~\cite{KiSh}
yields the regularity results for $f$ referred to in Theorem~\ref{thm:mainthm1} above.

\subsection{Uniform Domains and Co-Dimension Hausdorff Measures}

In this subsection we will gather together the geometric and measure-theoretic notions needed in discussing traces of 
functions in $D^{1,2}(Z_+)$. For a domain $U \subset Z$ denote the distance to the complement by 
$\delta_U(z)=d(z,Z \setminus U)$.

\begin{defn}\label{def:Uniform}
A domain $U\subset Z$ is $A$-uniform for some $A\geq 1$ if for every pair $x,y\in U$ there is a curve $\gamma \in U$ 
connecting $x$ and $y$ so that its length $\ell(\gamma)$ satisfies $\ell(\gamma)\leq Ad_Z(x,y)$ and for all $z\in \gamma$,
\[
\delta_{U}(z)\geq A^{-1}\min \{\ell(\gamma_{x,z}), \ell(\gamma_{y,z})\}.
\]
Here $\gamma_{x,z}$ and $\gamma_{y,z}$ are the subcurves of $\gamma$ connecting $z$ to $x$ and $y$ 
respectively. Such a curve is called an $A$-uniform curve.
\end{defn}

Let $\tau>0$ and $K\subset Z$. The co-dimension $\tau$ Hausdorff measure of $K$ is
\begin{equation}\label{eq:co-dim-tau-Hausdorff}
\mathcal{H}^{*,\tau}(K):=\lim_{\eps\to 0^+}
\inf\bigg\lbrace \sum_{i\in I\subset\N}\frac{\mu_a(B_i)}{\rad(B_i)^\tau}\, :\, A\subset\bigcup_{i\in I}B_i\, 
\text{ and }\rad(B_i)<\eps\bigg\rbrace.
\end{equation}

\subsection{Heat-Kernel Associated with a Dirichlet Form}

Corresponding to the Dirichlet form $\mathcal{E}_X$, there is a heat kernel $p_t:X\times X\to[0,\infty)$, $t>0$, 
see~\cite{BLS,FOT,G, S-C,Sta, Stb, St1, Stc}. While \cite{G} studied the structure of heat
kernels in the setting of Riemannian manifolds, the series of papers~\cite{Sta, Stb, St1, Stc} first studied
them in the setting of metric measure spaces equipped with a strongly local Dirichlet form such that the intrinsic
metric obtained from such a Dirichlet form gives a doubling measure supporting a $2$-Poincar\'e inequality.
From~\cite{KST} we know that the Dirichlet form obtained from a Cheeger differential structure as explained above
fits the hypotheses given in~\cite{St1}, and hence the results of~\cite{St1} apply here.

The heat kernel helps us construct solutions to the heat equation with initial data $f\in L^2(X)$. With
\[
P_tf(x):=\int_X f(w)p_t(x,w)d\mu_X(w),
\]
we know that whenever $v$ is a Lipschitz function on $X\times[0,\infty)$ with compact support in $X\times(0,\infty)$,
\[
\int_0^\infty\mathcal{E}_X(P_tf,v)\, dt+\int_X\int_0^\infty v(w,t)\, \partial_t P_tf(w)\, d\mu_X(w)\, dt=0.
\]
Embedded in the above claim is also the property that $t\mapsto P_t f(x)$ is differentiable with respect to almost every $t$ 
such that for each $x\in X$, the map $t\mapsto P_tf(x)$ is absolutely continuous on compact subintervals of $(0,\infty)$.
In~\cite{FOT} the heat operator $P_t$ is denoted by $T_t$, but we follow the notation from~\cite{BLS,G} here.
The following lemma gathers together results from~\cite{St1}.

\begin{lemma}\label{lem:heatkernel} We have the following properties for the heat kernel:
\begin{enumerate}
\item Markovian property:
\[
\int_X p_t(x,z) ~d\mu_X(z) = 1.
\]
\item Sub-Gaussian bounds:
\begin{equation}
\frac{1}{C_1\mu_X(B(x,\sqrt{t}))} e^{-\frac{c_1d(x,z)^2}{t}} \leq p_t(x,z) \leq \frac{C_2}{\mu_X(B(x,\sqrt{t}))} e^{-\frac{c_2d(x,z)^2}{t}}.
\label{eq:hke}
\end{equation}
\item Symmetry:  $p_t(x,z)=p_t(z,x)$, for each $x,z$ in $X$.
\item The semigroup property holds: 
\begin{equation}
 p_{t+s}(x,z)=\int_X p_t(x,w) p_s(w,z) ~d\mu_X(w).
 \label{eq:CK}
\end{equation}
\end{enumerate}
\end{lemma}

\subsection{Spectral Theory}\label{sec:background}

The spectral approach to fractional Laplacians seems to have been first formulated in~\cite{ST}, see for instance
the discussion in~\cite[page~2]{BLS}.
Our presentation also closely follows \cite{G}, but a nice description of the relationship between Dirichlet forms and spectral theory
is also given in~\cite[page~17]{FOT}. 
We note that the results at the end of this section are covered in the Ahlfors regular case in \cite{G}. 
Specifically, we slightly simplify the proof of \cite[Corollary 5.5]{G} and demonstrate that their work also applies in 
the doubling context without Ahlfors regularity. 

The Dirichlet form $\mathcal{E}_X(u,v)$ as described at the beginning of this section 
defines a closed regular Dirichlet form with domain $\mathcal{F}_X= N^{1,2}(X)$. Here, the equality is in the sense that 
each function $f \in \mathcal{F}_X$ has an almost everywhere representative in $N^{1,2}$. 
While, as described above, $Z$ also is equipped with a Dirichlet form, we will apply spectral theory solely to $\mathcal{E}_X$.
The associated Laplacian is denoted 
$\Delta_X$ with $\mathcal{D}(\Delta_X) \subset N^{1,2}(X) \subset L^2(X)$. This operator is self-adjoint 
and has a spectral decomposition
\[
-\Delta_X = \int_0^\infty \lambda dE_\lambda,
\]
where $dE_\lambda$ is a projection valued measure, i.e. the spectral resolution. 
This integral can be made sense of through pairing, namely 
$\int (-\Delta_X) u v \,d\mu_X = \int_0^\infty \lambda \langle dE_\lambda(u),v\rangle$ for $u,v\in L^2(X)$. 
See e.g. \cite{BLS, FOT} or \cite[Chapter XI]{Y}
 for further discussion. 
In particular, if $u$ is in the domain of $\Delta_X$ and $v\in L^2(X)$, then
\[
-\int_X v(x)\Delta_Xu(x)\, d\mu_X(x)=\int_0^\infty \lambda dE_\lambda(u,v).
\]
Note that if $v$ also belongs to $N^{1,2}(X)$, then
\[
\int_Xv(x)\Delta_Xu(x)\, d\mu_X(x)=-\int_X\langle D_Xv(x), D_Xu(x)\rangle_x\, d\mu_X(x).
\]
The domain of the Laplacian $\Delta_X$, denoted $\mathcal{D}(\Delta_X)$, can also be described by the equation
$\mathcal{D}(\Delta_X) = \{u \in L^2(X) : \int_0^\infty \lambda^2 dE_\lambda(u,u)  <\infty\}$. 
The domain of the Dirichlet form, $\mathcal{D}(\mathcal{E}_X)=N^{1,2}(X)$, is 
identifiable with $\{u \in L^2(X) : \int_0^\infty \lambda dE_\lambda(u,u) < \infty\}$.
The heat semigroup $\{P_t\}_{t>0}$ is given by
\[
P_t  = \int_0^\infty e^{-\lambda t} dE_\lambda,
\]
which is a contraction on $L^2(X)$.
For $0<\theta<1$ the fractional Laplacian can be expressed as
\[
(-\Delta_X)^\theta = \int_0^\infty \lambda^\theta dE_\lambda,
\]
and so when $u$ is in the domain of $\Delta_X$ and $v\in L^2(X)$, we have 
\[
\int_X ((-\Delta_X)^\theta u)\, v\, d\mu_X=\int_0^\infty\lambda^\theta dE_\lambda(u,v).
\]
We will also use
\[
\mathcal{F}_\theta = \{u \in L^2(X) : \int_0^\infty \lambda^{\theta} dE_\lambda(u,u)< \infty\} 
= \mathcal{D}((-\Delta_X)^\frac{\theta}{2}).
\]

The energy $\mathcal{E}_X$ 
can be recovered through a regularization process by defining
\[
\mathcal{E}_{X,t}(f,f) = \frac{1}{t} \int_X (f-P_t f) f ~d\mu_X=\frac{1}{2t} \int_X \int_X |f(x)-f(w)|^2 p_t(x,w) d\mu_X(x) d\mu_X(w).
\]
For the last equality, see \cite[Section 4]{G}. Then sending $t\to 0$ gives
\[
\lim_{t\to 0} \mathcal{E}_{X,t}(f,f) = \mathcal{E}_{X}(f,f),
\]
with $\mathcal{F}_X=\{f \in L^2(X) \colon \sup_{t>0} \mathcal{E}_{X,t}(f,f) < \infty \}$. 
The benefit of the regularization is that $\mathcal{E}_{X,t}(f,f)$ is defined and finite for all $f\in L^2(X)$ and a fixed $t>0$.

There is a non-negative continuous function $\eta^\theta_t(s):(0,\infty) \to (0,\infty)$ for which
$$e^{-t\lambda^\theta}=\int_0^\infty \eta^\theta_t(s) e^{-\lambda s} ds.$$
Thus the heat kernel corresponding to the infinitesimal generator $(-\Delta_X)^\theta$ is given by 
\begin{equation}\label{eq:qtdef}
q_t(x,y) = \int_0^\infty \eta^\theta_t(s) p_s(x,y) ~ds.
\end{equation}
For these facts see \cite[Section 5]{G} or \cite[Chapter IX, Section 11]{Y}.

The following lemma gives the required estimates for $q_t$ in doubling metric measure spaces.

\begin{lemma}\label{lem:qtestimates} 
Suppose $(X,d_X,\mu_X)$ is connected and measure doubling and the conclusions of 
Lemma~\ref{lem:heatkernel} hold. Then there are constants $C_1,C_2>0$ for which
\begin{eqnarray}
q_t(x,y) & \leq & C_1 \frac{t}{d(x,y)^{\theta} \mu_X(B(x,d(x,y)))} \ \ \ \text{ when } t,s>0 \label{eq:qestabove} \\
q_t(x,y) & \geq & C_2 \frac{t}{d(x,y)^{\theta} \mu_X(B(x,d(x,y)))} \ \ \ \text{ when } d(x,y)^{\theta}\geq t > 0 \label{eq:qestbeloq}
\end{eqnarray}

\end{lemma}

\begin{proof}
From \cite[Equations 5.32 and 5.33]{G} we have for some constants $B_1,B_2$ 
\begin{eqnarray}
\eta_t^\theta(s) & \leq & B_1 \frac{t}{s^{1+\theta}} \ \ \ \text{ when } t,s>0 \label{eq:qtabove-use} \\
\eta_t^\theta(s) & \geq & B_2 \frac{t}{s^{1+\theta}} \ \ \ \text{ when } s^{\theta}\geq t > 0. \label{eq:qestbeloq-use}
\end{eqnarray}
The claim follows by substituting the bounds from Estimate \eqref{eq:volumegrowth} and Part~2 of 
Lemma~\ref{lem:heatkernel} to the Equation \eqref{eq:qtdef}. Indeed, both estimates follow by noting that the 
main contribution for the integral comes from $s \sim d(x,y)^{2\theta}$. For the upper bound, the remaining 
scales are bounded by a geometric sum and its largest term, and for the lower bound we can restrict to the 
interval with $s \in [d(x,y)^{\theta}/2,d(x,y)^{\theta}]$.
\end{proof}

Similarly, we can define 
\[
\mathcal{E}_{\theta,t}(f,f)=\frac{1}{t} \int_X (f-T_{t,\theta} f) f ~d\mu_X=\frac{1}{2t} \int_X \int_X |f(x)-f(y)|^2 q_t(x,y) d\mu_X d\mu_X
\]
 where $T_{t,\theta}$, $t>0$, is the semigroup related to $(-\Delta_X)^\theta$, given by
$T_{t,\theta}=\int_0^\infty e^{-t\lambda^\theta}\, dE_\lambda$.
We also note that $\mathcal{E}_{\theta,t}$ is monotone decreasing in $t$
and obtain a description of the domain of the fractional Laplacian by  
\[
\mathcal{D}((-\Delta_X)^{\frac{\theta}{2}})=\mathcal{F}_\theta=\{f \in L^2(X) \colon \lim_{t\to 0}\mathcal{E}_{t,\theta}(f,f)
=\sup_{t>0}\mathcal{E}_{t,\theta}(f,f) < \infty \}
\]
These claims follow from \cite[Section 4]{G}, when applied to the heat semigroup $T_{t,\theta}$.

Using this we can identify the Domain with the Besov space.

\begin{prop}\label{prop:besovenergy}
Let $f \in B_{2,2}^\theta(X) \cap L^2(X)$. Then there is a constant $C$ so that

$$\frac{1}{C}\mathcal{E}_{\theta}(f,f) \leq \|f\|_{B^\theta_{2,2}(X)}^2 \leq C \mathcal{E}_{\theta}(f,f)$$
and moreover $B_{2,2}^\theta(X) \cap L^2(X)= \mathcal{F}_\theta$. 
\end{prop}

\begin{proof} We follow the proof in \cite[Theorem 5.2]{G}. We have by Estimate~\eqref{eq:qestabove} that
\begin{eqnarray*}
\mathcal{E}_{\theta,t}(f,f) &=& \frac{1}{2t} \int_X \int_X |f(x)-f(y)|^2 q_t(x,y) d\mu_X d\mu_X \\
&\lesssim & \int_X \int_X\frac{1}{d(x,y)^{\theta} \mu_X(B(x,d(x,y)))} |f(x)-f(y)|^2 ~d\mu_X ~d\mu_X.
\end{eqnarray*}

This holds for any $f \in L^2(X)$, and all $t>0$. Thus, by the remark before the statement, we obtain 
$\frac{1}{C}\mathcal{E}_{\theta}(f,f) \leq \|f\|_{B^\theta_{2,2}}$ and that 
$B_{2,2}^\theta(X) \cap L^2(X) \subset \mathcal{F}_\theta$. 

Define $X^2_t = \{(x,y) \in X \times X : d(x,y)^{2\theta} \geq t > 0\}.$
The other direction follows by
\begin{eqnarray*} 
\mathcal{E}_{\theta,t}(f,f) &\gtrsim& \frac{1}{2t} \int_{X^2_t} |f(x)-f(y)|^2 q_t(x,y) d\mu_X d\mu_X \\
&\gtrsim& \int_{X^2_t}\frac{1}{d(x,y)^{2\theta} \mu_X(B(x,d(x,y)))} |f(x)-f(y)|^2 ~d\mu_X ~d\mu_X.
\end{eqnarray*}
Here, the comparability constants are independent of $t$. 
Sending $t\to 0$, we obtain

\[
\mathcal{E}_{\theta}(f,f)
\gtrsim \int_X \int_X\frac{1}{d(x,y)^{2\theta} \mu_X(B(x,d(x,y)))} |f(x)-f(y)|^2 ~d\mu_X ~d\mu_X,
\]
and thus  $C\mathcal{E}_{\theta}(f,f) \geq \|f\|_{B^\theta_{2,2}}$ for some constant $C$ and 
$\mathcal{F}_\theta  \subset B_{2,2}^\theta(X) \cap L^2(X)$. This concludes the proof.
\end{proof}

\subsection{Explicit Solution}

In this section we give a kernel for the solutions of the Dirichlet problem \eqref{eq:DP}. Let  $\Pa((x,y),z)$ be the kernel given by
\[
\Pa((x,y),z) = C_a y^{1-a}\int_0^\infty s^{\frac{a-3}{2}}e^{-\frac{y^2}{4s}} p_s(x,z) ~ds,
\]
where $p_s(\cdot, \cdot)$ is the heat kernel associated to $\Delta$ and 
\[
\frac{1}{C_a}=\int_0^\infty \tau^{\tfrac{a-3}{2}}e^{-1/4\tau}\, d\tau.
\]
For each $f\in B_{2,2}^{\theta}(X) \cap L^2(X)= \mathcal{F}_{\theta}(X)$ with $\theta=\tfrac{1-a}{2}$, $x \in X$, and $y>0$, we set 
\begin{equation}\label{eq:explicit}
\Pi_a f(x,y):=u(x,y) 
= \int_X f(z) \Pa((x,y),z) ~d\mu_X(z).
\end{equation}
In~\cite[Lemma 3.1]{BLS} and~\cite[(3.21)]{BGMN} this function is denoted by $U(x,y)$. In the following lemma we collect some 
properties for this function, which are mostly contained in \cite{BLS}, see also~\cite{BGMN}.

 Bochner differentiation is
the analog of weak derivative for Banach space-valued functions on intervals. A function $v:(0,\infty)\to N^{1,2}(X)$ 
is differentiable if there is a function $g:(0,\infty)\to N^{1,2}(X)$ such that whenever $\varphi:(0,\infty)\to\R$ is 
compactly supported and smooth, we have
the following integration by parts formula:
\[
\int_0^\infty \varphi^\prime(t) v(t)\, dt=-\int_0^\infty \varphi(t) g(t)\, dt,
\]
with the above integrals taken as the Bochner integrals, see~\cite{DU}. Such a function $g$ is said to be the 
derivative of $f$, denoted $\partial_t v$.

\begin{lemma}\label{lem:propertiesextension}
Suppose that $f\in B_{2,2}^{\theta}(X) \cap L^2(X)$ and $u(x,y)=\Pi_a f(x,y)$. 
\begin{enumerate}
\item The map $(x,y) \mapsto u(x,y)$ is continuous and $y\mapsto u(x,y)$ is smooth in $y$ for each $x\in X$.
\item The map $y\mapsto u(\cdot,y)$ is Bochner measurable and defines a three times continuously Bochner differentiable 
function from $(0,\infty)$ to $L^2(X)$.
\item $u(\cdot,y) \in \mathcal{D}(\Delta_X)$ for every $y\in (0,\infty)$ and 
\begin{equation}
\label{eq:EDP}
\begin{cases}
\Delta_a u =\left(\Delta_X+ \dfrac{\partial^2}{\partial y^2}+\dfrac{a}{y}\dfrac{\partial}{\partial y}\right) u=0  & \text{in} \quad  Z_+\\
u(\cdot,0)=f(\cdot).
\end{cases}
\end{equation}
\item $u(\cdot,y) \in N^{1,2}(X)$ for every $y\in(0,\infty)$ 
\item $u \in N^{1,2}_{loc}(Z_+)$.
\end{enumerate}
Here, by the statement $u(\cdot,0)=f(\cdot)$ we mean that $u(\cdot, y)\to f(\cdot)$ in $L^2(X)$ as $y\to 0^+$.
\end{lemma}

\begin{proof}
The continuity and differentiality follow from the definition together with Lemma \ref{lem:heatkernel} and dominated convergence. 

The second and third claims follow from~\cite[Lemma 3.1]{BLS}. Indeed, in Part~3 of the proof of~\cite[Lemma~3.1]{BLS} it is
shown that $u(\cdot, 0)\to f(\cdot)$ in $L^2(X)$ as $y\to 0^+$.
There the terminology is not explained in depth, but the techniques of dominated convergence 
yield that the map from $[0,\infty)$ to $L^2$  given by $y\mapsto u(\cdot,y)$ is continuous, and thus 
Bochner measurable. In  \cite{BLS} 
only the first two derivatives are explicitly computed, but further derivatives are simple to compute by the same techniques.

Since $u(x,y)$ is continuous and $u(\cdot,y)  \in \mathcal{D}(\Delta_X) \subset \mathcal{F}_X$, we have 
$u(\cdot,y) \in N^{1,2}(X)$ for every $y\in (0,\infty)$. 

By Remark \ref{rem:diff-v-BochDer} below we have that the pointwise $y$-derivatives and the Bochner derivatives coincide. 
Now, for $[T_1,T_2]$  with $0<T_1 <T_2$ using the equation from the third claim and integration by parts (using the 
Bochner derivatives) and the third claim of the lemma, we get
\begin{align}
\int_{X\times[T_1,T_2]}& \int_X (\partial_y u)^2 + \langle D_X u, D_X u \rangle d\mu_a\\
&= \int_{T_1}^{T_2} \int_X (\partial_y u)^2 + \langle -\Delta_X u, u \rangle y^a dy d\mu_X \nonumber \\
&= \int_{T_1}^{T_2} \int_X \partial_y (y^au\,\partial_y u) -u\, \partial_y(y^a \partial_y u) - \langle \Delta_X u, u \rangle y^a dy d\mu_X \nonumber \\
&= \int_{T_1}^{T_2} \int_X \partial_y (y^au\,\partial_y u) dy d\mu_X \nonumber \\
&= \int_X y^{T_2}\partial_y u (x,T_2) u(x,T_2)-y^{T_1}\partial_y u (x,T_1) u(x,T_1) d\mu_X. \label{eq:integratiobyparts}
\end{align}

Since $\partial_y u(\cdot,y), u(\cdot,y)\in L^{2}(X)$ by \cite{BLS}, we have 
$(\partial_y u)^2 + \langle D_X u, D_X u\rangle\in L^2_{\text{loc}}(Z_+)$.
Further, we have that for every $y>0$ by the third part, 
\[
\int_X \langle D_X u, D_X u \rangle d\mu_a = \int_X (-\Delta_X u) u\, d\mu_a = \int_X u\, (\partial_y^2 +a^{-1}\partial_y)u\, d\mu_a.
\] 
Therefore $y\to \int_X \langle D_X u(\cdot, y), D_X u(\cdot,y) \rangle d\mu_a$
is now differentiable in $y$, and moreover, is continuous, since $y\to u(\cdot,y)$ is of class $C^3$ in the Bochner sense. 

In order to conclude from these the final claim that $u \in N^{1,2}_{loc}(Z_+)$, we would like to apply 
Lemma~\ref{lem:tensor-N1p}. However, in order to do this, we need that $u$ has measurable and 
$L^2_{loc}(Z)$-integrable (weak) upper gradients in the $\R$ and the $X$ directions. Note, that the lemma is 
applied with $Y=[T_1,T_2]$ with the weighted measure $y^ad\lambda$ and $X$ as is. For us $\partial_y u$ 
is continuous, and thus measurable.

Since $y\to u(\cdot,y)$ is a continuous curve in $L^2$, we have that if we define, for $k\in\N$, 
$u_k(x,y)=u(x,\lfloor y\cdot k \rfloor / k)$, that $u_k(\cdot,y) \to u(\cdot,y)$ as $k\to \infty$ in $L^2(X)$ for each 
$y>0$. We have that $u_k(\cdot,y) \in N^{1,2}(X)$, and $u_k(\cdot,y)$ is constant for $y\in [l/k,(l+1)/k)$ and any 
$l\in \N$. Thus, $u_k$ has a piecewise constant, and thus measurable, minimal upper gradient in the 
$X$-direction. Call it $g_{u_k^x}$. Now, by construction of the Cheeger differential structure,

\[
\int_{T_1}^{T_2} \int_X g_{u_k^x}^2 d\mu_X y^adt \leq C \int_{T_1}^{T_2} \int_X \langle D_X u_k, D_X u_k \rangle.
\]
Since $y\to \int_X \langle D_X u, D_X u \rangle d\mu_a$ is continuous, we get that $g_{u_k^x}$ is 
uniformly bounded in $L^2(X\times [T_1,T_2],\mu_a)$. By taking convex combinations of tails of $u_k$, 
we get a sequence of $v_k \to u$ in $L^2(X\times [T_1,T_2],\mu_a)$ where $g_{v_k^x}$ converges in 
$L^2(X\times [T_1,T_2],\mu_a)$ to a function $g\in L^2(X\times [T_1,T_2],\mu_a)$. Then, finally, as 
$v_k\to u$ in $L^2$, we obtain for almost every $y>0$, that $g(\cdot,y)$ is $2$-weak upper gradient for $u$.
\end{proof}

\begin{remark}\label{rem:diff-v-BochDer}
Note that if for $\mu_X$-a.e.~$x\in X$ the real-valued function $y\mapsto u(x,y)$ is absolutely continuous on $[0,\infty)$,
then $\partial_y u(x,y)$ exists for almost every $y$ such that the above integration by parts formula holds. It follows then
that the Bochner derivative coincides with this real derivative.
\end{remark}

Finally, for functions in the Besov class we can strengthen and show that the extension $u \in D^{1,2}(Z_+)$.

\begin{prop}\label{prop:trace-energy} 
If $f\in B^{\theta}_{2,2}(X)\cap L^2(X)$, then $u=\Pi_a f \in D^{1,2}(Z_+)$, and
\[\|u\|_{D^{1,2}(Z_+)}=\int_{0}^{\infty} \int_X (\partial_y u)^2 + \langle D_X u, D_X u \rangle d\mu_a 
= \dfrac{2^{2\theta -1} \Gamma(\theta)}{\Gamma(1-\theta)} \mathcal{E}_\theta(f,f).\]
\end{prop}

\begin{proof} By Proposition \ref{prop:besovenergy} we have that $B^{\theta}_{2,2}(X)\cap L^2(X)=\mathcal{F}_\theta$ 
with comparable norms. It is easy to show that $\mathcal{D}((-\Delta_X)^{\theta})$ is dense in $\mathcal{F}_\theta$ 
in the norm $\mathcal{E}_\theta(f,f)$, and thus we may assume that $f \in \mathcal{D}((-\Delta_X)^{\theta})$.
From \cite[Lemma 3.2]{BLS}, we have that 
\[
\lim_{y\to 0^+}- \dfrac{2^{2\theta -1} \Gamma(\theta)}{\Gamma(1-\theta)}y^a\partial_y u = (-\Delta_X)^\theta f,
\]
where $\Gamma(\cdot)$ is the standard gamma function.

Moreover, $\lim_{y\to\infty} -y^a\partial_y u = 0$ weakly in $L^2(X)$. Indeed, for any $v\in L^2(X)$, by the 
proof of~\cite[Lemma 3.1]{BLS} we have that
\[
\int_X -y^a\partial_y u \, v \,d\mu_X  
= \frac{1}{\Gamma(s)} \int_0^\infty  \frac{y^{1+a}e^{-\frac{y^2}{4t}}}{2t} \int_X P_t (-\Delta_X)^\theta f v d\mu_X \frac{dt}{t^{1-\theta}}.
\]
Then, sending $y\to \infty$ easily proves the claim together with dominated convergence, and since 
$\|\int_X P_t (-\Delta_X)^\theta f v \,d\mu_X \| \leq \|(-\Delta_X)^\theta f\|_{L^2(X)} \|v\|_{L^2(X)}$.

 The claim then follows from these combined with the calculation at the end of the proof of 
 Lemma~\ref{lem:propertiesextension} and  sending $T_1\to 0$ and $T_2 \to \infty$.
\end{proof}

\section{Tensorization}

In this section we collect some results regarding tensorization that will be useful in working with $Z=X\times\R$.
Recall the discussion from Remark~\ref{remk:Cartesian-PI} above. We treat a more general Cartesian product here
by considering the product of two doubling metric measure spaces
$(X,d_X,\mu_X)$, $(Y,d_Y,\mu_Y)$, both supporting a $2$-Poincar\'e 
inequality. The product $Z=X\times Y$ is equipped with the metric $d_Z$ given by
\[
d_Z((x_1,y_1),(x_2,y_2))=\sqrt{d_X(x_1,x_2)^2+d_Y(y_1,y_2)^2}
\]
and the measure $\mu_Z$ given by $d\mu_Z(x,y)=d\mu_X(x)\, d\mu_Y(y)$. While in our application we will consider
$Y=\R$ with $d\mu_Y(y)=y^a\, dy$, we formulate the results in the section in this generality as they are of independent 
interest and add to the results in~\cite{APS}.

\subsection{Tensorization of Newton-Sobolev Energies}

\begin{lemma}\label{lem:tensor-N1p}
If $f\in L^2(Z)$, then 
there is a modification of $f$ on a $\mu_X\times\mu_Y$-measure zero subset of $Z$ that 
is in $N^{1,2}(Z)$ if and only if, after modification on a set of $\mu_X\times\mu_Y$-measure zero if necessary,
we have 
\begin{enumerate}
\item for $\mu_X$-almost every $x\in X$ we have that $f^x:=f(x,\cdot)\in N^{1,2}(Y)$ with 
$(x,y)\mapsto g_{f^x}(y)\in L^2(Z)$ where $g_{f^x}$ is an upper gradient of $f^x$ in $X$, and
\item for $\mu_Y$-almost every $y\in Y$ we have that $f^y:=f(\cdot, y)\in N^{1,2}(X)$ with
$(x,y)\mapsto g_{f^y}(x)\in L^2(Z)$.
\end{enumerate}
\end{lemma}

If either $(X,\mu_X)$ or $(Y,\mu_Y)$ does not support s $2$-Poincar\'e inequality, we do not know the validity
of the conclusion in the above lemma, see for example the discussion in~\cite{GR}.

\begin{proof}
Since the two metrics $d_Z$ and $\max\{d_X,d_Y\}$ are biLipschitz equivalent, 
and the space $N^{1,2}(Z)$ is biLipschitz invariant,
in this proof we will assume that
\[d_Z((x_1,y_1),(x_2,y_2))=\max\{d_X(x_1,x_2),d_Y(y_1,y_2)\}.\] For $r>0$ and $(x_0,y_0)\in Z$, consider 
$B_Z:=B_Z((x_0,y_0),r)=B_X(x_0,r)\times B_Y(y_0,r)$. Suppose that $f$ satisfies the hypotheses of the two conditions 
of the lemma. Then by the respective Poincar\'e inequalities of $X$ and $Y$, we obtain
\begin{align*}
\vint_{B_Z}&\vint_{B_Z}|f(x_1,y_1)-f(x_2,y_2)|\, d\mu_Z(x_1,y_1)d\mu_Z(x_2,y_2)\\
  &=\vint_{B_X}\vint_{B_X}\vint_{B_Y}\vint_{B_Y}|f(x_1,y_1)-f(x_2,y_2)|\, d\mu_Y(y_2)d\mu_Y(y_1)\, d\mu_X(x_2)d\mu_X(x_1)\\
  &\le \vint_{B_X}\vint_{B_X}\vint_{B_Y}\vint_{B_Y}|f(x_1,y_1)-f(x_1,y_2)| \\
  &\ \ \ \ \ \ \ \ \ \  + \ |f(x_1,y_2)-f(x_2,y_2)|\,\,d\mu_Y(y_2) d\mu_Y(y_1) d\mu_X(x_2)d\mu_X(x_1)\\
  &\le 2C_Y r \vint_{B_X}\left(\vint_{\lambda_YB_Y}g_{f^{x_1}}^2\, d\mu_Y\right)^{1/2}\, d\mu_X(x_1)
        \\
        &\ \ \ \ \ \ \ \ \ \  + \  2C_Yr\vint_{B_Y}\left(\vint_{\lambda_XB_X}g_{f^{y_2}}^2\, d\mu_X\right)^{1/2}\, d\mu_Y(y_2)\\
  &\le C_0 r\left(\vint_{\lambda_0B_X}\vint_{\lambda_0B_Y}[g_{f^x}(y)+g_{f^y}(x)]^2\, d\mu_X(x)d\mu_Y(y)\right)^{1/2}.
\end{align*}
In the above, $\lambda_0=\max\{\lambda_X,\lambda_Y\}$ with $C_X,\lambda_X$ the
constants related to the Poincar\'e inequality on $X$ and $C_Y,\lambda_Y$ the corresponding constants for $Y$, 
and $C_0$ is a constant depending on $C_X,C_Y,\lambda_X,\lambda_Y$ and doubling.
It follows that with $g(x,y):=g_{f^x}(y)+g_{f^y}(x)$, we have that $g\in L^p(Z)$ and that the pair
$(f,g)$ satisfies the $p$-Poincar\'e inequality with constants $C_0$ and $\lambda_0$. Hence, by
\cite[Theorem~10.3.4]{HKSTbook}, we know that there is a constant $C\ge 1$ (with $C$ depending only on the
Poincar\'e constants of $Z$ and on $C_0,\lambda_0$) such that $Cg$ is a $p$-weak upper gradient of a function
$f_0$ in $Z$ such that $f_0=f$ $\mu_a$-a.e.~in $Z$. Let $E$ consist of points $(x,y)\in Z$ for which
$f_0(x,y)\ne f(x,y)$. Then by Fubini's theorem we have that for $\mu_X$-a.e.~$x\in X$, we have that
$\mu_Y(E\cap \{x\}\times Y)=0$, and so for such $x\in X$ we have that $f(x,\cdot)=f_0(x,\cdot)$ $\mu_Y$-a.e.~in $Y$.
Combining this with the assumptions on $f$, we have that for $\mu_X$-a.e.~$x\in X$,
the $2$-capacity (with respect to $Y$) of $E\cap\{x\}\times Y$ is zero. Similarly, for $\mu_Y$-a.e.~$y\in Y$, the
$2$-capacity (with respect to $X$) of $E\cap X\times\{y\}$ is zero.
If $Y=\R$, then we can say more; indeed, only the empty set is of zero $2$-capacity with respect to $(\R, |y|^a\, dy)$.
It follows that for $\mu_X$-a.e.~$x\in X$ we have that $f(x,y)=f_0(x,y)$ whenever $y\in\R$.

Conversely, suppose that $f\in N^{1,2}(Z)$, and let $g$ be an upper gradient of $f$ such that $g\in L^2(Z)$. Then by
Fubini's theorem, for $\mu_X$-a.e.~$x\in X$ we know that $g(x,\cdot),f(x,\cdot)\in L^2(Y)$ and for $\mu_Y$-a.e.~$y\in Y$ we have
$g(\cdot, y),f(\cdot, y)\in L^2(X)$. For such $x$ and $y$ we know that $f^x\in N^{1,2}(Y)$ with $g(x,\cdot)$ 
acting as an upper gradient
of $f^x$, and that $f^y\in N^{1,2}(X)$ with $g(\cdot, y)$ acting as an upper gradient of $f^y$. That is, $f$ satisfies the two
conditions listed in the claim of the lemma.
\end{proof}

\begin{remark}
In the above, we only require that $f^y$ have an upper gradient $g_{f^y}$ such that the map $(x,y)\mapsto g_{f^x}(y)$ belongs
to $L^2(Z)$, and that $f^x$ satisfies a similar condition. We do not require $g_{f^y}$ to be a \emph{minimal weak upper gradient}
if $f^y$, since we would not in practice be able to guarantee the measurability of such function in $Z$.
However, if $Y=\R$ is equipped
with the weighted measure $|y|^a\, dy$, then we have the following improvement, see Corollary~\ref{cor:Banach} below. 
\end{remark}

To prove the corollary mentioned above, we need the following lemma.

\begin{lemma}\label{lem:cylinder-modulus}
For $A\subset X$ that is $\mu_X$-measurable, and $h>0$, consider the family $\Gamma(A,h)$ of curves
in $Z$ of the form $\gamma_x:[0,h]\to Z$ given by $\gamma_x(t)=(x,t)$; $x\in A$. Then 
\[
\Mod_2(\Gamma(A,h))\simeq \frac{\mu_X(A)}{h^{1-a}}.
\]
Moreover, if $u\in N^{1,2}_{{\rm loc}}(Z_+)$ and $0\le t_1<t_2<\infty$, then for $\mu_X$-a.e.~$x\in X$ we have
that $u\circ\gamma_x:[t_1,t_2]\to\R$ is absolutely continuous with $\partial_t u\circ\gamma_x\le g_u\circ\gamma_x$,
where $\gamma_x:[t_1,t_2]\to Z_+$ is given by $\gamma_x(t)=(x,t)$. In particular,  if $\text{Cap}_2(A\times\{0\})=0$,
then $\mu_X(A)=0$.
\end{lemma} 

\begin{proof}
Setting $\rho:=h^{-1}\chi_{A\times[0,h]}$ we see that for each $\gamma\in\Gamma(A,h)$, $\int_\gamma\rho\, ds=1$.
It follows that 
\[
\Mod_2(\Gamma(A,h))\le \int_Z\rho^2\, d\mu_a=\frac{\mu_X(A)}{(a+1) h^{1-a}}.
\]
On the other hand, if $\rho$ is a non-negative Borel measureable function on $Z$ such that 
for each $\gamma\in\Gamma(A,h)$ we have $\int_\gamma\rho\, ds\ge 1$, then
\begin{align*}
1\le \int_0^h\rho(x,t)\, dt &\le \left(\int_0^h\rho(x,t)^2 t^a\, dt\right)^{1/2}\left(\int_0^ht^{-a}\, dt\right)^{1/2}\\
   &=\left(\int_0^h\rho(x,t)^2 t^a\, dt\right)^{1/2}\left(\frac{h^{1-a}}{1-a}\right)^{1/2}.
\end{align*}
It follows that
\[
\frac{1-a}{h^{1-a}}\le \int_0^h\rho(x,t)^2t^a\, dt.
\]
Integrating over $x\in A$ gives
\[
\frac{(1-a)\mu_X(A)}{h^{1-a}}\le \int_Z\rho^2\, d\mu_a.
\]
Taking the infimum over all such $\rho$ gives
\[
\frac{(1-a)\mu_X(A)}{h^{1-a}}\le \Mod_2(\Gamma(A,h)).
\]

As a consequence of the above, we know that $\Mod_2(\Gamma(A,h))=0$ if and only if $\mu_X(A)=0$. 
The final statement about capacity follows from the fact that $\Mod_2$ of the family of all curves intersecting a
set of measure zero is null if and only if the capacity of that set is zero, see~\cite{ HKSTbook,N}.
\end{proof}

\begin{cor}\label{cor:Banach}
Let $Y=\R$ or $Y=(0,\infty)$ be equipped with the Euclidean metric and the measure $\mu_Y$ given by $d\mu_Y=y^a\, dy$. 
Let $u\in N^{1,2}(Z)$. Then for $\mu_X$-almost every $x\in X$ we have that $u(x,\cdot)$ is absolutely continuous on $Y$
with $y\mapsto \partial_y u(x,y)\in L^2(Z)$. In particular, $y\mapsto \partial_y u(x,y)$ is measurable on $Z$.
\end{cor}

\begin{proof}
Note that as $u\in N^{1,2}(Z)$, $u$ is absolutely continuous on $2$-modulus almost every non-constant compact rectifiable curve
in $Z$, see for instance~\cite{ HKSTbook,N}. Hence by Lemma~\ref{lem:cylinder-modulus} we know that for $\mu_X$-almost
every $x\in X$ the map $y\mapsto u(x,y)$ is absolutely continuous on compact subintervals of $Y$. As $u$ has an upper gradient
$g$ in $Z$, we know that $|\partial_y u(x,y)|\le g(x,y)$, and so as $g\in L^2(Z)$, it suffices to show that
$\partial_y u$ is measurable on $Z$. To this end we argue as follows.

 Suppose $Y=\bbR$ so that $u\colon X\times \bbR\to \bbR$. The proof is similar if $Y=(0,\infty)$. Let
 \[
 \widetilde{X}=\{x\in X\, :\, y\mapsto u(x,y) \mbox{ is continuous}\}.
 \]
 Then $\widetilde{X}$ is a measurable subset of $X$ and $\mu(X\setminus \widetilde{X})=0$. 
 We show that the domain of $\partial_{y}u(x,y)$ is a measurable subset of $\widetilde{X}\times \bbR$ and that 
 $\partial_{y}u(x,y)$ is a measurable function of $(x,y)$ on that set. It suffices to prove that for every closed 
 bounded interval of positive length $B\subset \mathbb{R}$, the set
 \[D_{B}:=\{(x_{0},y_{0})\in \widetilde{X}\times \bbR\, :\, \partial_{y}u(x_{0},y_{0}) \mbox{ exists and belongs to }B\}\]
is measurable. 

Fix a closed bounded interval $B$ of positive length. For $(x,y)\in X\times \bbR$ it follows from compactness of 
$B$ that $\partial_{y}u(x_{0},y_{0})$ exists and belongs to $B$ if and only if for every 
$\varepsilon \in \mathbb{Q}^{+}$ there exists $\delta \in \mathbb{Q}^{+}$ and $q\in B\cap \mathbb{Q}$ such that
\[
\sup_{0<|t|<\delta} \frac{|u(x_{0},y_{0}+t)-u(x_{0},y_{0})-qt|}{|t|}\leq \varepsilon.
\]
Observe that if $x_{0}\in \widetilde{X}$ it is equivalent to take the above supremum only for 
$t\in (-\delta,\delta)\cap \mathbb{Q}$. We claim that
\begin{equation}\label{measeq}
D_{B}=\bigcap_{\varepsilon \in \mathbb{Q}^{+}}\bigcup_{\delta \in \mathbb{Q}^{+}} \bigcup_{q\in B\cap \mathbb{Q}} \bigcap_{t\in \mathbb{Q}\cap (-\delta, \delta)\setminus \{0\}} D_{B}(\varepsilon,\delta,q,t),
\end{equation}
where
\[
D_{B}(\varepsilon,\delta,q,t)=\bigg\lbrace(x_{0},y_{0})\in \widetilde{X}\times \bbR \, :\, 
 \frac{|u(x_{0},y_{0}+t)-u(x_{0},y_{0})-qt|}{|t|}\leq \varepsilon\bigg\rbrace.
\]
That the left side of~\eqref{measeq} is contained in the right side follows from routine approximation by rationals. 
To show the right side is contained in the left side, we take $\varepsilon=1/n$ for $n\in \mathbb{N}$ to find 
a sequence of rational $\delta_{n}>0$ and $q_{n}\in B\cap \mathbb{Q}$ such that for all 
$t\in \mathbb{Q}\cap (-\delta, \delta)\setminus \{0\}$ we have 
\[\frac{|u(x_{0},y_{0}+t)-u(x_{0},y_{0})-q_{n}t|}{|t|}\leq 1/n.\]
This inequality necessarily holds for all $t\in (-\delta, \delta)\setminus\{0\}$ because the left side is a 
continuous function of $t$ away from $0$. Since $B$ is compact, the sequence $\{q_{n}\}$ has a 
subsequence converging to some $q\in B$. It is then routine to show that $\partial_{y}u(x_{0},y_{0})$ exists and equals $q$.

Since $u$ is measurable on $X\times \bbR$ and $\widetilde{X}$ is a measurable subset of $X$, it 
follows that each set $D_{B}(\varepsilon,\delta,q,t)$ is measurable. Hence, by \eqref{measeq}, $D_B$ is a 
measurable subset of $X\times \bbR$.
\end{proof}

\subsection{Tensorization of Differentiable Structures}\label{Subsect:diff-struct-tensor}

Recall the definition of a Cheeger differentiable structure from Definition~\ref{deff:Ch-Diff-Struct}. Given two doubling metric measure
spaces $(X,d_X,\mu_X)$ and $(Y,d_Y,\mu_Y)$, in this section we construct a differential structure on $Z=X\times Y$.

Let $(X,d_{X},\mu_{X})$ and $(Y,d_{Y},\mu_{Y})$ be metric measure spaces, each equipped with a doubling measure 
and admitting a differentiable structure. Denote the charts in the differentiable structure by $(U_{i}, \varphi_{i})$ and 
$(V_{j},\psi_{j})$ where $\varphi_{i} \colon X\to \bbR^{m_{i}}$ and $\psi_{j} \colon X\to \bbR^{n_j}$ respectively for $i, j\in \bbN$.

\begin{defn}
Suppose $f\colon X\times Y\to \bbR$ is Lipschitz with respect to $d_{X}\times d_{Y}$ and 
$(x_{0},y_{0})\in U_{i}\times V_{j}$ for some $i,j\in \bbN$. 

We say $f$ is differentiable at the point $(x_{0},y_{0})$ in the $X$ direction with derivative $D_{X}f(x_{0},y_{0})$ (with 
respect to $\varphi_{i}$) if $D_{X}f(x_{0},y_{0})$ is the unique element $v\in \bbR^{m_{i}}$ such that
\[\lim_{x\to x_{0}} \frac{|f(x,y_{0})-f(x_{0},y_{0})-v\cdot (\varphi_{i}(x)-\varphi_{i}(x_{0})|}{d_{X}(x,x_{0})}=0.\]
We say $f$ is differentiable at $(x_{0},y_{0})$ in the $Y$ direction with derivative $D_{Y}f(x_{0},y_{0})$ (with respect 
to $\psi_{j}$) if $D_{Y}f(x_{0},y_{0})$ is the unique element $w\in \bbR^{n_{j}}$ such that
\[\lim_{x\to x_{0}} \frac{|f(x_{0},y)-f(x_{0},y_{0})-w\cdot (\psi_{j}(y)-\psi_{j}(y_{0})|}{d_{Y}(y,y_{0})}=0.\]
\end{defn}

For each $i,j\in \bbN$, let $\varphi_{i}\times \psi_{j}\colon X\times Y\to \mathbb{R}^{m_{i}}\times \mathbb{R}^{n_{j}}$ 
denote the map \[(x,y)\mapsto (\varphi_{i}(x), \psi_{j}(y)).\]

In this section we will prove the following theorem.

\begin{thm}\label{diffproduct}
Suppose $(X,d_{X},\mu_{X})$ and $(Y,d_{Y},\mu_{Y})$ are doubling metric measure spaces each admitting a 
differentiable structure. Then $(X \times Y,d_X \times d_Y, \mu_{X} \times \mu_{Y})$ is also a doubling metric 
measure space and admits a differentiable structure. 

More precisely, the charts in the differentiable structure for $(X \times Y,d_X \times d_Y, \mu_{X} \times \mu_{Y})$ 
can be chosen to be $(U_{i}\times V_{j}, \varphi_{i}\times \psi_{j})$ for $i,j \in \bbN$, where 
$\varphi_{i}\times \psi_{j}\colon U_{i}\times V_{j}\to \bbR^{m_{i}}\times \bbR^{n_{j}}$ is defined by 
$(\varphi_{i}\times \psi_{j})(x,y)=(\varphi_{i}(x), \psi_{j}(y))$. 

Given $f\colon X\times Y\to \bbR$ and $i,j\in \bbN$, for almost every $(x_{0},y_{0})\in U_{i}\times V_{j}$, the 
derivative of $f$ with respect to $\varphi_{i}\times \psi_{j}$ is $(D_{X}f(x_{0},y_{0}), D_{Y}f(x_{0},y_{0}))$, 
where $D_{X}f(x_{0},y_{0})$ and $D_{Y}f(x_{0},y_{0})$ are the derivatives of $f$ with respect to $X$ and 
$\varphi_{i}$ or $Y$ and $\psi_{j}$.
\end{thm}

We now proceed to prove Theorem~\ref{diffproduct}. To prove Theorem~\ref{diffproduct}, we first establish 
some needed lemmas. We denote the Lipschitz constant of any Lipschitz map $g$ by $L_{g}$. Fix a 
Lipschitz map $f\colon X\times Y \to \bbR$.

Recall that a set $P$ in a metric space $M$ is porous if there exists $\lambda>0$ such that for every 
$x\in P$ the following property holds. There exists $x_{n}\in M$ such that $x_{n}\to x$ and 
$B(x_{n},\lambda d(x_{n},x))\cap P=\varnothing$. A set is $\sigma$-porous if it is a countable union of porous sets.

\begin{lemma}
Given $m, n\in \bbN$, $v\in \bbR^m, w\in \bbR^n, \varepsilon, \delta \in (0,1)$, let $P(v,w,\varepsilon, \delta,f)$ be 
the set of $(x_{0},y_{0})\in X\times Y$ with the following properties.
\begin{enumerate}
\item For every chart $U_{i}$ containing $x_{0}$ with $m_{i}=m$, $0<d_{X}(x,x_{0})<\delta$ implies
\[|f(x,y_{0})-f(x_{0},y_{0})-v\cdot (\varphi_{i}(x)-\varphi_{i}(x_{0}))|<\varepsilon d_{X}(x,x_{0}).\]
\item For every chart $V_{j}$ containing $y_{0}$ with $n_{j}=n$, $0<d_{Y}(y,y_{0})<\delta$ implies
\[|f(x_{0},y)-f(x_{0},y_{0})-w\cdot (\psi_{j}(y)-\psi_{j}(y_{0}))|<\varepsilon d_{Y}(y,y_{0}).\]
\item There exist charts $U_{i}$ containing $x_{0}$ with $m_{i}=m$ and $V_{j}$ containing $y_{0}$ with $n_{j}=n$ 
for which there is $(x,y)$ arbitrarily close to $(x_{0},y_{0})$ with
\begin{align}\label{eq:bad-approach}
&|f(x,y)-f(x_{0},y_{0})-v\cdot (\varphi_{i}(x)-\varphi_{i}(x_{0}))-w\cdot (\psi_{j}(y)-\psi_{j}(y_{0}))|\\
& > \varepsilon (2L_{f}+2+|w|L_{\psi_{j}})d_{X}(x,x_{0})+\varepsilon d_{Y}(y,y_{0}). \notag
\end{align}
\end{enumerate}
Then the set $P(v,w,\varepsilon, \delta,f)$ is porous.
\end{lemma}

\begin{proof}
Fix a set $P=P(v,w,\varepsilon, \delta,f)$ as in the statement and let $(x_{0}, y_{0})\in P$. Fix charts 
$U_{i}$ and $V_{j}$ as in (3). Let $(x,y)$ satisfy the estimates $0<d_{X}(x,x_{0})<\delta/2$ and 
$0<d_{Y}(y,y_{0})<\delta/2$ such that the inequality~\eqref{eq:bad-approach} in (3) holds. 
Notice that $U_{i}$ is a chart containing $x_{0}$ with $m_{i}=m$ and $V_{j}$ is a chart containing 
$y_{0}$ with $n_{j}=n$. Notice 
\[d_{X\times Y}((x_{0},y_{0}),(x,y_{0}))=d_{X}(x,x_{0}).\]
Hence to show that $P$ is porous it suffices to show that
\begin{equation}\label{willbeporous}
P\cap B_{X\times Y}((x,y_{0}),\varepsilon d_{X}(x,x_{0}))=\varnothing.
\end{equation}
Suppose not and fix $(a,b)\in P\cap B_{X\times Y}((x,y_{0}),\varepsilon d_{X}(x,x_{0}))$. We show how to obtain a contradiction.

By using $(x_{0},y_{0})\in P$ and applying property (1) from the definition of $P$, we see
\begin{equation}\label{p1}
|f(x,y_{0})-f(x_{0},y_{0})-v\cdot (\varphi_{i}(x)-\varphi_{i}(x_{0}))|<\varepsilon d_{X}(x,x_{0}).
\end{equation}`
Next notice that $d(b,y)\leq d(b,y_{0})+d(y_{0},y)<\delta$. Hence we can use $(a,b)\in P$ and property~(2) 
from the definition of $P$ to obtain
\begin{equation}\label{p2}
|f(a,y)-f(a,b)-w\cdot (\psi_{j}(y)-\psi_{j}(b))|<\varepsilon d_{Y}(y,b).
\end{equation}
Using the fact that $(a,b)\in B_{X\times Y}((x,y_{0}),\varepsilon d_{X}(x,x_{0}))$, it is straightforward to 
derive the following estimates
\begin{equation}\label{p3}
|f(a,y)-f(x,y)|\leq L_{f}\varepsilon d_{X}(x,x_{0}),
\end{equation}
\begin{equation}\label{p4}
|f(x,y_{0})-f(a,b)|\leq L_{f}\varepsilon d_{X}(x,x_{0}),
\end{equation}
\begin{equation}\label{p5}
|\psi_{j}(b)-\psi_{j}(y_{0})|\leq L_{\psi_{j}}\varepsilon d_{X}(x,x_{0}).
\end{equation}
\begin{equation}\label{p6}
d_{Y}(y,b)\leq d_{Y}(y,y_{0})+d_{Y}(y_{0},b)\leq d_{Y}(y,y_{0})+d_{X}(x,x_{0}).
\end{equation}
Combining \eqref{p1}--\eqref{p6} and applying the triangle inequality, it can be shown that
\begin{align*}
&|f(x,y)-f(x_{0},y_{0})- v\cdot (\varphi_{i}(x)-\varphi_{i}(x_{0})) - w\cdot (\psi_{j}(y)-\psi_{j}(y_{0}))|\\
&\qquad \leq \varepsilon (2L_{f}+2+|w|L_{\psi_{j}})d_{X}(x,x_{0})+\varepsilon d_{Y}(y,y_{0}).
\end{align*}
This contradicts the fact that $(x_{0},y_{0})$ together with $(x,y)$ satisfies the inequality~\eqref{eq:bad-approach} of
(3) in the definition of $P$. Therefore \eqref{willbeporous} is true. This shows that  $P$ is porous as claimed.
\end{proof}

Now define
\[
P(f)=\bigcup_{\substack{m\in \bbN\\ n\in \bbN}} \bigcup_{\substack{v\in \mathbb{Q}^{m}\\w\in \mathbb{Q}^{n}} }\bigcup_{\substack{\varepsilon \in (0,1)\cap \mathbb{Q}\\ \delta \in (0,1)\cap \mathbb{Q} }} P(v,w,\varepsilon,\delta,f).
\]
Clearly $P(f)\subset X\times Y$ is a $\sigma$-porous set. Since $\mu_X\times\mu_Y$ is doubling, it follows that 
$\mu_Z(P(f))=0$. This is because doubling measures assign measure zero to porous sets, which is well known 
and follows from the fact that the Lebesgue differentiation theorem holds whenever the underlying measure is 
doubling. For an explicit proof one could follow the steps in \cite{PS17}, which do not depend on the Carnot 
group structure in that paper.

\begin{lemma}\label{diff}
Fix $(x_{0},y_{0})\in X\times Y$ such that
\begin{enumerate}
\item For a chart $(U_{i}, \varphi_{i})$, $f$ is differentiable at $(x_{0},y_{0})$ in the $X$ direction with unique 
derivative $D_{X}f(x_{0},y_{0})$.
\item For a chart $(V_{j}, \psi_{j})$, $f$ is differentiable at $(x_{0},y_{0})$ in the $Y$ direction with unique 
derivative $D_{Y}f(x_{0},y_{0})$.
\item $(x_{0},y_{0})\notin P$.
\end{enumerate}
Then $f$ is differentiable with respect to the chart $(U_{i}\times V_{j}, \varphi_{i}\times \psi_{j})$ on $X\times Y$ 
with unique derivative $(D_{X}f(x_{0},y_{0}),D_{Y}f(x_{0},y_{0}))$.
\end{lemma}

\begin{proof}
Let $\varepsilon>0$ be rational. By (1) and (2) from the statement of the lemma, we can choose $\delta>0$ 
rational such whenever $0<d_{X}(x,x_{0})<\delta$ and $0<d_{Y}(y,y_{0})<\delta$,
\[|\,f(x,y_{0})-f(x_{0},y_{0})-D_{X}f(x_{0},y_{0})\cdot (\varphi_{i}(x)-\varphi_{i}(x_{0}))\,|<\varepsilon d_{X}(x,x_{0})/2,\]
\[|\,f(x_{0},y)-f(x_{0},y_{0})-D_{Y}f(x_{0},y_{0})\cdot (\psi_{j}(y)-\psi_{j}(y_{0}))\,|<\varepsilon d_{Y}(y,y_{0})/2.\]
Fix $v\in \mathbb{Q}^{m_{i}}$ and $w\in \mathbb{Q}^{n_{j}}$ such that
\[|\,v-D_{X}f(x_{0},y_{0})\,|\leq \varepsilon/2L_{\varphi_{i}}\]
and
\[|\,w-D_{Y}f(x_{0},y_{0})\,|\leq \varepsilon/2L_{\psi_{j}}.\]
It is easy to check using the triangle inequality that (1) and (2) from the definition of $P(v,w,\varepsilon,\delta,f)$ hold. 
Since $(x_{0}, y_{0})\notin P$ it follows that (3) cannot hold. Hence 
\begin{align*}
&|\, f(x,y)-f(x_{0},y_{0})-v\cdot (\varphi_{i}(x)-\varphi_{i}(x_{0}))-w\cdot (\psi_{j}(y)-\psi_{j}(y_{0}))\,|\\
& \qquad \leq \varepsilon (2L_{f}+2+|w|L_{\psi_{j}})d_{X}(x,x_{0})+\varepsilon d_{Y}(y,y_{0}).
\end{align*}
Using again the triangle inequality gives for $0<d(x,x_{0})<\delta$ and $0<d(y,y_{0})<\delta$
\begin{align*}
&|\, f(x,y)-f(x_{0},y_{0})-D_{X}f(x_{0},y_{0})\cdot (\varphi_{i}(x)-\varphi_{i}(x_{0})) \\
&\hskip6cm -D_{Y}f(x_{0},y_{0})\cdot (\psi_{j}(y)-\psi_{j}(y_{0}))\, |\\
& \leq \varepsilon (2L_{f}+3+|w|L_{\psi_{j}})d_{X}(x,x_{0})+2\varepsilon d_{Y}(y,y_{0}).\\
& \leq \varepsilon (3L_{f}+4+|D_{Y}f(x_{0},y_{0})|L_{\psi_{j}})d_{X}(x,x_{0})+2\varepsilon d_{Y}(y,y_{0}).
\end{align*}
This shows that $f$ is differentiable at $(x_{0},y_{0})$ with derivative with respect to the given chart given by 
$(D_{X}f(x_{0},y_{0}), D_{Y}f(x_{0},y_{0}))$.

To show uniqueness of the derivative, suppose $v\in \bbR^{m_{i}}$ and $w\in \bbR^{n_{j}}$ such that
\[
\lim_{(x,y)\to (x_{0}, y_{0})} \frac{|f(x,y)-f(x_{0},y_{0})-v\cdot (\varphi_{i}(x)-\varphi_{i}(x_{0}))
-w\cdot (\psi_{j}(y)-\psi_{j}(y_{0}))|}{d_{X}(x,x_{0})+d_{Y}(y,y_{0})}=0.
\]
Substituting $y=y_{0}$ gives
\[\lim_{x\to x_{0}} \frac{|f(x,y_{0})-f(x_{0},y_{0})-v\cdot (\varphi_{i}(x)-\varphi_{i}(x_{0}))|}{d_{X}(x,x_{0})}=0.\]
Uniqueness of the derivative $D_{X}f(x_{0},y_{0})$ in the $X$ direction yields $v=D_{X}f(x_{0},y_{0})$. 
Similarly substituting $x=x_{0}$ yields $w=D_{Y}f(x_{0},y_{0})$. This proves uniqueness of the derivative 
$(D_{X}f(x_{0},y_{0}),D_{Y}f(x_{0},y_{0}))$.
\end{proof}

\begin{lemma}\label{Fubini}
Let $(U_{i}, \varphi_{i})$ and $(V_{j}, \psi_{j})$ be charts on $X$ and $Y$ respectively. Then both 
$D_{X}f(x_{0},y_{0})$ and $D_{Y}f(x_{0},y_{0})$ exist for $\mu_{X}\times \mu_{Y}$ almost every 
$(x_{0}, y_{0})\in U_{i}\times V_{j}$.
\end{lemma}

\begin{proof}
For each fixed $y_{0}\in V_{j}$, $D_{X}f(x_{0},y_{0})$ exists for $\mu_{X}$ almost every $x_{0}\in U_{i}$ because 
$x\mapsto f(x,y_{0})$ is a Lipschitz map $X\to \bbR$. The claimed result will follow for $D_{X}f(x_{0},y_{0})$ by 
Fubini's theorem once we show that the set where $D_{X}f(x_{0},y_{0})$ exists is a measurable subset of $X\times Y$ 
and that $(x_{0},y_{0})\mapsto D_{X}f(x_{0},y_{0})$ is a measurable function on that set. To that end, first notice 
that we can write
\[
\{(x_{0},y_{0})\in U_{i}\times V_{j}: D_{X}f(x_{0},y_{0}) \mbox{ exists with respect to } \varphi_{i}\}=\bigcup_{N=1}^{\infty} A_{N},
\]
where $A_{N}$ is the set of $(x_{0}, y_{0})\in U_{i}\times V_{j}$ for which $D_{X}f(x_{0},y_{0})$ exists with 
respect to $\varphi_{i}$ and belongs to the closed ball $\overline{B(0,N)}\subset \mathbb{R}^{m_{i}}$. We claim 
$A_{N}=\widetilde{A}_{N}$, where
\begin{align*}
\tilde{A}_{N}=\bigcap_{\substack{\varepsilon > 0\\ \varepsilon \in \mathbb{Q}}} \bigcup_{\substack{\delta>0\\ \delta \in \mathbb{Q}}} \bigcup_{\substack{q\in \mathbb{Q}^{m_{i}}\\|q|\leq N}} &\{(x_{0},y_{0})\in U_{i}\times V_{j}:\\
&\sup_{0<d_{X}(x,x_{0})<\delta} \frac{|f(x,y_{0})-f(x_{0},y_{0})-q\cdot (\varphi_{i}(x)-\varphi_{i}(x_{0}))|}{d_{X}(x,x_{0})}\leq \varepsilon \}.
\end{align*}
Clearly $A_{N}\subset \widetilde{A}_{N}$. For the opposite inclusion, suppose $(x_{0},y_{0})\in \widetilde{A}_{N}$. 
Then there exists $\delta_{n}\downarrow 0$ and a sequence $q_{n}\in \overline{B}(0,N)$ such that
\[
\sup_{0<d_{X}(x,x_{0})<\delta_{n}} \frac{|f(x,y_{0})-f(x_{0},y_{0})-q_{n}\cdot (\varphi_{i}(x)-\varphi_{i}(x_{0}))|}{d_{X}(x,x_{0})}<\frac{1}{n}.
\]
By taking a subsequence if necessary, we may assume that $q_{n}\to q\in \overline{B}(0,N)$. Then the triangle inequality yields
\[
\sup_{0<d_{X}(x,x_{0})<\delta_{n}} \frac{|f(x,y_{0})-f(x_{0},y_{0})-q\cdot (\varphi_{i}(x)-\varphi_{i}(x_{0}))|}{d_{X}(x,x_{0})}
<\frac{1}{n} + |q_{n}-q|L_{\varphi_{i}}.
\]
Since the right side converges to $0$ as $n\to \infty$, this shows $D_{X}f(x_{0},y_{0})$ exists and equals 
$q\in \overline{B}(0,N)$. Hence $(x_{0},y_{0})\in A_{N}$, which shows that $A_{N}=\widetilde{A}_{N}$. The fact that 
$\widetilde{A}_{N}$ is measurable follows because $X$ is separable, so one can equivalently consider the 
supremum over a countable dense set of $x$ with $0<d_{X}(x,x_{0})<\delta$.

A similar argument with $\overline{B}(0,N)$ replaced by other balls shows that $D_{X}f(x_{0},y_{0})$ is measurable on its 
domain. The argument for $D_{Y}f(x_{0},y_{0})$ is analogous with $X$ replaced by $Y$.
\end{proof}

We can now prove Theorem \ref{diffproduct}.

\begin{proof}[Proof of Theorem \ref{diffproduct}]
The measure $\mu_{X}\times \mu_{Y}$ is doubling, so porous sets have measure zero. Hence for any Lipschitz 
$f\colon X\times Y\to \bbR$ we have $\mu_{X}\times \mu_{Y}(P(f))=0$. Fix a chart $(U_{i},\varphi_{i})$ and 
$(V_{j},\psi_{j})$. By applying Lemma \ref{Fubini}, we see that almost every point $(x_{0},y_{0})\in U_{i}\times V_{j}$ has 
the following properties:
\begin{enumerate}
\item Both $D_{X}f(x_{0},y_{0})$ and $D_{Y}f(x_{0},y_{0})$ exist and are unique.
\item $(x_{0},y_{0})\notin P$.	
\end{enumerate}
Hence, by Lemma \ref{diff}, $f$ is differentiable at $(x_{0},y_{0})$ with respect to the chart 
$(U_{i}\times V_{j}, \varphi_{i}\times \psi_{j})$ with unique derivative $(D_{X}f(x_{0},y_{0}),D_{Y}f(x_{0},y_{0}))$.
\end{proof}

\begin{remark}\label{rm:Da-deff}
When $Y=\R$ equipped with the Euclidean metric and the measure $y^a\, dy$, the corresponding tensorization of the differential
structure $D_X$ on $X$ and the Euclidean differential structure on $Y$ yields the differential structure on $Z$ denoted by
$D_a$ and the corresponding Dirichlet form is denoted $\Ead$. This is not to be confused with the fractional 
Dirichlet form $\mathcal{E}_\theta$ referred to in~\eqref{eq:frac-Dirich-1} above. 
\end{remark}

\begin{lemma}\label{lem:tensor-innerprod-uppergrad}
With the differential structure $D_a$ as in Remark~\ref{rm:Da-deff}, we consider the inner product on this structure given by
\[
\langle D_au(x,y),D_av(x,y)\rangle_{(x,y)}:=\langle D_Xu(x,y),D_Xv(x,y)\rangle_x+\partial_yu(x,y)\, \partial_y v(x,y).
\]
Then there is a constant $C>0$ such that whenever $u$ is a Lipschitz function on $Z$, we have
\[
\frac{1}{C} g_u(x,y)^2\le \langle D_au(x,y),D_au(x,y)\rangle_{(x,y)}\le C\, g_u(x,y)^2
\]
where $g_u$ is the minimal weak upper gradient of $u$ (equivalently, $g_u=\Lip u$ $\mu_a$-a.e.~in $Z$). 
\end{lemma}

\begin{proof}
From~\cite[Theorem~3.4]{APS} we know that
\[
\Lip u(x,y)^2=\limsup_{w\to x}\frac{|u(w,y)-u(x,y)|^2}{d_X(w,x)^2}+|\partial_y u(x,y)|^2
\]
for $\mu_a$-a.e.~$(x,y)\in Z$ when $u$ is locally Lipschitz continuous. Combining this with the results from~\cite{Ch}
completes the proof.
\end{proof}

\section{Potential Theory in Relation to the Domain \texorpdfstring{$Z_+\subset Z$}{Zplus}}\label{Sect:Traces}

We fix $-1<a<1$ and set $\theta= \frac{1-a}{2}$. 
The central function space on $X$ from our point of view is the Besov space $B^\theta_{2,2}(X)$. 
Recall from~\eqref{eq:BesovEnergy} that this space consists of functions
$f\in L^1_{loc}(X)$ such that 
\[
\Vert f\Vert_{B^\theta_{2,2}(X)}^2:=\int_X\int_X\frac{|f(y)-f(x)|^2}{d_X(y,x)^{2\theta}\mu_X(B(y,d_X(y,x)))}\, d\mu_X(x)\, d\mu_X(y)
\]
is finite.

Recall Remark~\ref{remk:Cartesian-PI}.
It is known from the work of Maly~\cite{Maly} that when $\Omega$ is a bounded domain in a metric measure space 
equipped with a doubling measure supporting a Poincar\'e inequality and $\partial\Omega$ is equipped with a co-dimension
$\tau$-Hausdorff measure induced by the measure on $\Omega$, and if $\Omega$ is a uniform domain, then the
trace space of the Sobolev space $N^{1,2}(\Omega)$ is $B^\theta_{2,2}(\partial\Omega)\cap L^2(X)$ for a suitable choice of $\theta$.
In our setting, the domain we are interested in is $Z_+$, but even when $X$ is bounded, $Z_+$ is unbounded. 
We will show in this section that $Z_+$ is a uniform
domain in $Z$ and that a homogeneous version of Maly's theorem~\cite{Maly} holds true. This gives the desired link between
the Dirichlet space $D^{1,2}(Z_+)$ and $B^\theta_{2,2}(X)$ where $\theta=(1-a)/2$.
We first need the following proposition, which also
demonstrates the co-dimensionality between  the measures $\mu_a$ on $Z_+$ and
$\mu_X$ on $X$.

\begin{prop}\label{prop:unif+PI}
Fix $-1<a<1$. Then the 
domain $Z_+:=X\times (0,\infty)$ is a uniform domain in $Z$, and for each $x\in X$ and $r>0$, we have
\begin{equation}\label{eq:co-dim-measure}
\mu_X(B(x,r))\simeq \frac{\mu_a(B_Z((x,0), r))}{r^{1+a}}. 
\end{equation}
\end{prop}

\begin{proof} 
The metric $d_\infty$ on $Z$ given by
$d_\infty((x_1,y_2),(x_2,y_2))=\max\{d_X(x_1,x_2),|y_1-y_2|\}$ is biLipschitz with respect to the metric $d_Z$, and it is 
convenient to consider this metric here. So in this proof $d_Z$ will denote $d_\infty$. 

Note that $Z_+$ cannot be a John domain even if $X$ is itself bounded, for there is
no reasonable point acting as the John center. However, we show that $Z_+$ is a uniform domain. 
Since $\mu_X$ is doubling on $X$ and supports a $2$-Poincar\'e 
inequality,  we know that there is a constant $C_q\ge 1$ such that $X$ is $C_q$-quasiconvex. 

Let $(x_1,y_1), (x_2,y_2)\in Z_+$. Without loss of generality, we assume that $0<y_1\le y_2$. Let
$\beta_1:[0,y_2-y_1+d(x_1,x_2)]\to Z_+$ be given by $\beta_1(t)=(x_1,y_1+t)$, and 
$\beta_3:[0,d(x_1,x_2)]\to Z_+$ be given by $\beta_3(t)=(x_2,d(x_1,x_2)+y_2-t)$. We also fix a
$C_q$-quasiconvex curve $\widehat{\beta_2}:[0,L]\to X$ in $X$ that starts from $x_1$ and ends at $x_2$, and we
set $\beta_2:[0,L]\to Z_+$ by $\beta_2(t)=(\widehat{\beta_2}(t),d(x_1,x_2)+y_2)$. The concatenation of
these three curves, $\beta_1+\beta_2+\beta_3$, is denoted by $\gamma$. Then the trajectory of $\gamma$ 
lies entirely in $Z_+$, starts at $(x_1,y_1)$, and ends at $(x_2,y_2)$. We wish to show that $\gamma$ is a 
uniform curve. 

If $(x,y)\in \beta_1$, then $x=x_1$ and
\[
\distz((x,y), X)=y\ge y-y_1=\ell(\beta_1\vert_{[0,t]})=\ell(\gamma_{(x_1,y_1),(x,y)}).
\]
If $(x,y)\in\beta_3$, then $x=x_2$ and
\[
\distz((x,y),X)=y\ge y-y_2=\ell(\gamma_{(x_2,y_2),(x,y)}).
\]
If $(x,y)\in \beta_2$, then $y=d(x_1,x_2)+y_2$ and so
\[
\distz((x,y),X)=d(x_1,x_2)+y_2.
\]
On the other hand,
\[
\ell(\gamma_{(x,y),(x_2,y_2)})=\ell((\widehat{\beta_2})_{x,x_2})+d(x_1,x_2)\le (1+C_q)d(x_1,x_2).
\]
Therefore in this case,
\[
(1+C_q)\, \distz((x,y),X)\ge \ell(\gamma_{(x,y),(x_2,y_2)}).
\]
Moreover, 
\begin{align*}
\ell(\gamma)=\ell(\beta_1)+\ell(\widehat{\beta_2})+\ell(\beta_3)
    &=(y_2-y_1+d(x_1,x_2))+\ell(\widehat{\beta_2})+d(x_1,x_2)\\
    &\le (y_2-y_1)+(2+C_q)d(x_1,x_2)\\
    &\le 2(2+C_q)d_Z((x_1,y_1),(x_2,y_2)).
\end{align*}
As $1+C_q\le 2(2+C_q)$, it follows that $\gamma$ is a $2(2+C_q)$-uniform curve in $\Omega$ with end points
$(x_1,y_1)$ and $(x_2,y_2)$. Thus $Z_+$ is a uniform domain.

Finally, we prove the codimensionality condition that is the last part of the proposition.

For $x_0\in X$ and $r>0$ note that $B_Z((x_0,0),r)\cap Z_+=B_X(x_0,r)\times(0,r)$, and so
\begin{equation}\label{eq:co-dim}
 \mu_a(B_Z((x_0,0),r)\cap Z_+)=\frac{r^{1+a}}{1+a} \mu_X(B_X(x_0,r)).
\end{equation}
If $A\subset X=X\times\{0\}=\partial Z_+$ is a Borel set and $\tau>0$, then from~\eqref{eq:co-dim-tau-Hausdorff}
we know that the co-dimension $\tau$ Hausdorff
measure $\mathcal{H}^{*,\tau}(A)$ is the number
\[
\lim_{\eps\to 0^+}\inf\bigg\lbrace
 \sum_i\frac{\mu_a(B_i)}{\rad(B_i)^\tau}\, :\, A\subset \bigcup_i B_i\text{ and }\rad(B_i)<\eps
 \bigg\rbrace. 
\]
Here, $B_i$ are balls in $\overline{Z_+}$, centered at points in $X\times\{0\}$; 
this is the ``co-dimension $\tau$ Hausdorff measure on $X$. Then by~\eqref{eq:co-dim}, we have
\begin{align*}
\mathcal{H}^{*,1+a}(A)&=\frac{1}{1+a} \lim_{\eps\to 0^+}\inf\bigg\lbrace
 \sum_i\mu_X(B_i^X)\, :\, A\subset \bigcup_i B_i^X\text{ and }\rad(B_i^X)<\eps
 \bigg\rbrace\\ &=\frac{1}{1+a}\mu_X(A)
\end{align*}
because $\mu_X$ is a Borel regular measure on $X$. Thus $\mu_X$ is a codimension $1+a$ Hausdorff measure
on $X=\partial Z_+$ satisfying~\eqref{eq:co-dim}.
\end{proof}

Recall the definition of $D^{1,2}(Z_+)$ from Definition~\ref{deff:SobolevN-D}, and that with $-1<a<1$ fixed,
we set $\theta=\tfrac{1-a}{2}$.

\begin{prop}\label{prop-Trace}
Trace space of $D^{1,2}(Z_+)$ is the class $B^\theta_{2,2}(X)$ where $\theta=\tfrac{1-a}{2}$. That is, the 
operator $T$ as defined in~\eqref{eq:Trace} forms a bounded linear operator
\[
 T:D^{1,2}(Z_+)\to B^\theta_{2,2}(X).
 \]
 In addition, there is a bounded linear operator
 \[
  E:B^\theta_{2,2}(X)\to D^{1,2}(Z_+)
 \]
 such that $T\circ E$ is the identity map on $B^\theta_{2,2}(X)$. Furthermore,
 for functions $u$ in $D^{1,2}(Z_+)$ that have continuous extensions, also denoted by $u$, 
 to $X\times\{0\}$, we have $Tu=u\vert_{X\times\{0\}}$.
\end{prop}

The proof of this proposition follows along the lines of the proof of~\cite[Theorem~1.1]{Maly}; however, 
our uniform domain $Z_+$ is not bounded and the functions $u$ in the above proposition are not 
necessarily in the global Sobolev class $N^{1,2}(Z_+)$ as non-zero constant functions have globally finite
energy but are not in $N^{1,2}(Z_+)$. Therefore the proof in~\cite{Maly} is not directly applicable here.
Instead, for the convenience of the reader, we provide a complete proof below.

\begin{proof}
For $x\in X$ we set $\gamma_x:[0,\infty)\to\overline{Z_+}$ by $\gamma_x(t)=(x,t)$.
Let $u\in D^{1,2}(Z_+)$ with a $2$-weak upper gradient $g\in L^2(Z_+)$.
As $Z_+$ is a uniform domain, by~\cite{BjSh} we know that $u$ has an extension, 
also denoted $u$, that belongs to $D^{1,2}(\overline{Z_+})$, and so $\text{Cap}_2$-almost every
point in $\overline{Z_+}$ is a $\mu_a$-Lebesgue point of $u$.
Let $x,w\in X$ be two distinct points such that $u\circ\gamma_x$ and $u\circ\gamma_w$ are absolutely
continuous and have limits along $\gamma_x$, $\gamma_w$ at $x$ and $w$ respectively, see 
Lemma~\ref{lem:cylinder-modulus}.  By Lemma~\ref{lem:cylinder-modulus}
and by the connection between $\text{Cap}_2(E)=0$ and $2$-modulus of the family of all curves in $\overline{Z_+}$
intersecting $E$ being zero (see~\cite{HKSTbook} or~\cite{N}), we know that
if $E\subset \partial Z_+=X$ has
$\text{Cap}_2(E)=0$, then $\mu_X(E)=0$. 
Thus the trace defined by~\eqref{eq:Trace} exists $\mu_X$-a.e.
Hence, we may assume also that $x,w$ are both Lebesgue
points of $u$ (with respect to the measure $\mu_a$). 

From the proof of Proposition~\ref{prop:unif+PI}, the concatenation of the three curves $\gamma_1$, $\gamma_2$,
and $\gamma_3$ is a uniform curve in $Z_+$ with end points $(x,0)$ and $(w,0)$, where
\begin{align*}
\gamma_1(t)&:=(x,t), \ \ \ 0\le t\le d_X(x,w),\\
\gamma_3(t)&:=(w,d_X(x,w)-t),\ \ \ 0\le t\le d_X(x,w),\\
\gamma_2(t)&:=(\beta(t),d_X(x,w)), \ \ \ 0\le t\le \ell_X(\beta), 
\end{align*}
where $\beta$ is a $C_q$-quasiconvex curve in $X$ arc-length parametrized to be beginning at $x$ and ending at $w$.
We cover $\gamma$ by Whitney-type balls as follows. Let 
$N_0$ be the smallest positive integer such that $N_0 d_Z(x,w)/2\ge \ell_X(\beta)$. Because $\beta$ is $C_q$-quasiconvex,
it follows that $N_0\le 2C_q$.

To construct the cover, we provide three groups of balls, one for each of $\gamma_1$, $\gamma_2$, and
$\gamma_3$. The first group, associated with $\gamma_1$, consists of balls $B_k$ for negative integers $k$.
The middle group, $B_k$ for $k=0,\cdots,N_0$, is associated with $\gamma_2$, while the balls
$B_k$ for positive integers $k>N_0$ are associated with $\gamma_3$. This is done as follows.

For $k=0,\cdots, N_0-1$ let
$B_k:=B(\gamma_2(kd_X(x,w)),d_X(x,w)/2)$ the ball centered a length-distance $kd_X(x,w)/2$ along $\gamma_2$
from $(x,d_X(z,w))$, with radius $d_X(w,x)/2$, and let $B_{N_0}:=B(\gamma_2(\ell(\beta)),d_X(w,x)/2)$.
For positive integers $k>N_0$ we set 
$B_k:=B(\gamma_3(t_k),2^{N_0-k}d_X(x,w))$ for $t_k:=(1-2^{N_0-k})d_X(x,w)$. For negative integers $k$
we set $B_k:=B(\gamma_1(\tau_k),2^k d_X(x,w))$ with $\tau_k:=2^kd_X(x,w)$. 
For $k\in\Z$ we set $r_k=2^{-|k|}d_X(x,w)$. Note that the radius of $B_k$ is comparable to $r_k$. 

Strictly speaking, the balls $B_k$ with $k\le 0$ or $k\ge N_0$ are not balls centered at $x,w$ respectively;
however, balls centered at $x,w$ respectively and twice the radius of $B_k$ will contain $B_k$, and
so as long as $\lim_{\rho\to 0^+} \rho \left(\vint_{B(x,\rho)} g^2\, d\mu_a\right)^{1/2}=0$, by the Poincar\'e 
inequality we still have that $\lim_{k\to-\infty} u_{B_k}=u((x,0))$ and $\lim_{k\to\infty} u_{B_k}=u((w,0))$.
Hence,  by a telescoping argument (see for example~\cite{HKSTbook, HeiK}), we have (we denote $u(x,0)$ by an
abuse of notation as $u(x)$) and for a choice of $\eps>0$, 
\begin{align*}
|u(x)-u(w)|\le \sum_{k\in\Z} |u_{B_k}-u_{B_{k+1}}|&\le C \sum_{k\in\Z}\vint_{2B_k}|u-u_{2B_k}|\, d\mu_a\\
  &\le C \sum_{k\in\Z} r_k \left(\vint_{2B_k}g^2\, d\mu_a\right)^{1/2}\\
  &\le C\sum_{k\in\Z} r_k^{\tfrac{1-a-\eps}{2}} \left(r_k^{1+a+\eps}\vint_{2B_k}g^2\, d\mu_a\right)^{1/2},
\end{align*}
where we use the doubling property of $\mu_a$ together with the fact that $B_{k+1}\subset 2B_k$ for $k\in\Z$,
followed by the $2$-Poincar\'e inequality on $Z_+$.
Now by H\"older's inequality, we obtain
\begin{align*}
|u(x)-u(w)|&\le C\left(\sum_{k\in\Z}r_k^{1-(a+\eps)}\right)^{1/2}
  \left(\sum_{k\in\Z}r_k^{1+a+\eps}\vint_{2B_k}g^2\, d\mu_a\right)^{1/2}\\
 &\le C d_X(x,w)^{\tfrac{1-a-\eps}{2}} 
\left(\sum_{k\in\Z}\frac{r_k^{1+a+\eps}}{\mu_a(B_k)}\int_{2B_k}g^2\, d\mu_a\right)^{1/2}\\
&\le C d_X(x,w)^{\tfrac{1-a-\eps}{2}} \left(\sum_{k\in\Z} \frac{r_k^\eps}{\mu_X(B(z_k,r_k))} \int_{2B_k} g^2\, d\mu_a\right)^{1/2},
\end{align*}
where $z_k=x$ when $k\le 0$ and $z_k=w$ when $k>0$.
Following~\cite{Maly}, we set
\[
C_1[x,w]:=\bigcup_{k=0}^\infty 2B_k,\ \ \ C_2[x,w]:=\bigcup_{k=1}^\infty 2B_{-k}.
\]
Then we have
\begin{align*}
|u(x)-u(w)|\le C d_X(x,w)^{\tfrac{1-a-\eps}{2}}&
 \bigg(\int_{C_1[x,w]}\frac{d_Z((x,0),z)^\eps}{\mu_X(B(x,d_Z((x,0),z)))}g(z)^2\, d\mu_a(z)\\
 +\int_{C_2[x,w]}&\frac{d_Z((w,0),z)^\eps}{\mu_X(B(w,d_Z((w,0),z)))}g(z)^2\, d\mu_a(z)\bigg)^{1/2}.
\end{align*}
Recalling that $\theta=\tfrac{1-a}{2}$, we obtain
\begin{align}\label{eq:part-Besov-term}
\frac{|u(x)-u(w)|^2}{d_X(x,w)^{2\theta}}\le C d_Z(x,w)^{-\eps}  &
 \bigg(\int_{C_1[x,w]}\frac{d_Z((x,0),z)^\eps}{\mu_X(B(x,d_Z((x,0),z)))}g(z)^2\, d\mu_a(z)\notag \\
 +\int_{C_2[x,w]}&\frac{d_Z((w,0),z)^\eps}{\mu_X(B(w,d_Z((w,0),z)))}g(z)^2\, d\mu_a(z)\bigg).
\end{align}
Now,  we estimate the Besov energy seminorm of $u$ on $\partial Z_+=X\times\{0\}\simeq X$.

By the results in~\cite[Theorem~1.11]{JJRRS},  we know that $u$ is Borel measureable in $\overline{Z_+}$,
and hence its restriction to $\partial Z_+=X$ is measurable with respect to the Borel measure $\mu_X$.
Recall that
\[
\Vert u\Vert_{B^\theta_{2,2}(X)}^2
 =\int_X\int_X\frac{|u(x)-u(w)|^2}{d_X(x,w)^{2\theta}\mu_X(B(x,d_X(x,w)))}\, d\mu_X(w)\, d\mu_X(x).
\]
By~\eqref{eq:part-Besov-term}, we obtain
\begin{align}
\Vert u\Vert_{B^\theta_{2,2}(X)}^2
&\lesssim \int_X\! \int_X\!\!\bigg(\int_{C_1[x,w]}\frac{d_Z((x,0),z)^\eps\, d_X(x,w)^{-\eps}}{\mu_X(B(x,d_Z((x,0),z)))\mu_X(B(x,d_X(x,w)))}g(z)^2 d\mu_a(z)\notag \\
 +\int_{C_2[x,w]}&\frac{d_Z((w,0),z)^\eps\, d_X(x,w)^{-\eps}}{\mu_X(B(w,d_Z((w,0),z)))\mu_X(B(x,d_X(x,w)))}g(z)^2\, d\mu_a(z)\bigg)d\mu_X(w)\, d\mu_X(x)\notag \\
 & \qquad =:C(I_1+I_2), \label{eq:Besov-energy}
\end{align}
where, by Tonelli's theorem,
\begin{align}
I_1&:=\!\!\int_{X^2}\!\int_{C_1[x,w]}\frac{g(z)^2\,d_Z((x,0),z)^\eps\, d_X(x,w)^{-\eps}}{\mu_X(B(x,d_Z((x,0),z)))\mu_X(B(x,d_X(x,w)))}\, d\mu_Z(z)\, d\mu_X(w)\, d\mu_X(x) \notag\\
 =&\int_{Z_+}g(z)^2\int_{X^2}\frac{\chi_{C_1[x,w]}(z)\,d_Z((x,0),z)^\eps\, d_X(x,w)^{-\eps}}{\mu_X(B(x,d_Z((x,0),z)))\mu_X(B(x,d_X(x,w)))}\, d\mu_X(w)\, d\mu_X(x)\, d\mu_a(z), \label{eq:term1}\\
I_2&:=\!\!\int_{X^2}\!\int_{C_2[x,w]}\!\frac{g(z)^2\,d_Z((w,0),z)^\eps\, d_X(x,w)^{-\eps}}{\mu_X(B(w,d_Z((w,0),z)))\mu_X(B(x,d_X(x,w)))}\, d\mu_Z(z)\, d\mu_X(w)\, d\mu_X(x)\notag\\
 =&\int_{Z_+}g(z)^2\!\int_{X^2}\frac{\chi_{C_2[x,w]}(z)\,d_Z((w,0),z)^\eps\, d_X(x,w)^{-\eps}}{\mu_X(B(w,d_Z((w,0),z)))\mu_X(B(x,d_X(x,w)))}\, d\mu_X(w)\, d\mu_X(x)\, d\mu_a(z). \label{eq:term2}
\end{align}
Here, we abbreviated $\int_X \int_X$ by $\int_{X^2}$.

Observe that if $z\in C_1[x,w]$, then $\tfrac{d_Z((x,0),z)}{4}\le \text{dist}_Z(z,\partial Z_+)\le 4d_X(x,w)$, and
if $z\in C_2[x,w]$, then $\tfrac{d_Z((w,0),z)}{4}\le \text{dist}_Z(z,\partial Z_+)\le 4d_X(x,w)$. 
To simplify notation, we set $\delta(z):=\distz(z,\partial Z_+)$ for $z\in Z_+$. 
Then, 
\[
\chi_{C_1[x,w]}(z)\le \chi_{B(z, 4\delta(z))}((x,0))\chi_{X\setminus B(x,\delta(z)/4)}(w).
\]
Thus, for each $z\in Z_+$, setting $B_k:=B(x,2^k\delta(z)/4)\subset X$ for $k=0,1,\cdots$, we have
\begin{align*}
\int_X&\int_X\frac{\chi_{C_1[x,w]}(z)\,d_Z((x,0),z)^\eps\, d_X(x,w)^{-\eps}}{\mu_X(B(x,d_Z((x,0),z)))\mu_X(B(x,d_X(x,w)))}\, d\mu_X(w)\, d\mu_X(x)\\
\lesssim &\int_{B(z,4\delta(z))\cap \partial Z_+}\int_{X\setminus B(x,\delta(z)/4)}
  \frac{\delta(z)^\eps d_X(x,w)^{-\eps}}{\mu_X(B(x,\delta(z)))\, \mu_X(B(x, d_X(x,w)))}\, d\mu_X(w)\, d\mu_X(x)\\
\lesssim &\int_{B(z,4\delta(z))\cap\partial Z_+}\frac{\delta(z)^\eps}{\mu_X(B(x,\delta(z)))}
 \sum_{k=0}^\infty \frac{(2^k\delta(z))^{-\eps}}{\mu_X(B(x,2^k\delta(z)/4))}\mu_X(B_{k+1}\setminus B_k)\, d\mu_X(x)\\
 \lesssim &\int_{B(z,4\delta(z))\cap\partial Z_+}\frac{\delta(z)^\eps}{\mu_X(B(x,\delta(z)))}
 \sum_{k=0}^\infty 2^{-k\eps}\delta(z)^{-\eps}\ d\mu_X(x)\ \lesssim 1.
\end{align*}
A similar argument shows that 
\[
\int_X\int_X\frac{\chi_{C_2[x,w]}(z)\,d_Z((w,0),z)^\eps\, d_X(x,w)^{-\eps}}{\mu_X(B(w,d_Z((w,0),z)))\mu_X(B(x,d_X(x,w)))}\, 
  d\mu_X(w)\, d\mu_X(x)\lesssim 1.
\]
Therefore, from~\eqref{eq:term1} and~\eqref{eq:term2} we obtain
$I_1+I_2\le C\int_{Z_+}g(z)^2\, d\mu_Z(z)$,
from whence we obtain by~\eqref{eq:Besov-energy} that
\begin{equation}\label{eq:Bes-Dir-bound}
\Vert u\Vert_{B^\theta_{2,2}(X)}^2\le C\int_{Z_+} g^2\, d\mu_Z
\end{equation}
as desired.

We next show that if $u:X\to\R$ such that $\Vert u\Vert_{B^\theta_{2,2}(X)}$ is finite, then $u$ has an extension to $Z_+$,
also denoted $u$, such that $u$ has a $2$-weak upper gradient $g$ in $Z_+$ with the property that $g\in L^2(Z_+)$.
To this end, for each $n\in\Z$ we choose $A_n\subset X$
to be a maximal $2^{-n}$-separated set. We can ensure also that $A_n\subset A_{n+1}$ for each $n\ge 0$.  
As $\mu_X$-almost
every point in $X$ is a Lebesgue point of $u$, we can also ensure that each point in $A_n$ is a Lebesgue point of $u$. 
Then the balls 
$Q_{i,n}:=B_X(x_i,2^{-n})\times(3\cdot 2^{-n},7\cdot 2^{-n})\subset Z_+$, $x_i\in A_n$, forms a Whitney cover of $Z_+$ such that
\begin{enumerate}
\item we have $Z_+=\bigcup_{i,n}Q_{i,n}$,
\item for each $\tau\ge 1$ there is a constant $C_\tau\ge 1$ such that for each $n\in\N$ we have
$\sum_i\chi_{\tau Q_{i,n}}\le C_\tau$, that is, we have the bounded overlap property,
\item for each $i_0,n_0$ there are at most $M$ number of balls $Q_{i,n}$ such that $Q_{i,n}\cap Q_{i_0,n_0}\ne \emptyset$,
and moreover, if $Q_{i,n}$ intersects $Q_{i_0,n_0}$, then $|n-n_0|\le 1$.
\item The center $z_{i,n}:=(x_i, 5\cdot 2^{-n})\in Z_+$ of $Q_{i,n}$ satisfies $\delta(z_{i,n})=5\cdot 2^{-n}$ and
$Q_{i,n}$ is at a distance $3\cdot 2^{-n}$ from $\partial Z_+$.
\end{enumerate}
Here, for $\tau>0$, by $\tau Q_{i,n}$ we mean the scaled ball $B_X(x_i,\tau 2^{-n})\times((5-2\tau) 2^{-n},(5+2\tau) 2^{-n})$.
It is sometimes in our interest to make sure that $0<\tau<5/2$ so that $\overline{\tau Q_{i,n}}\subset Z_+$.
However, when $\tau>5/2$ we have that $\tau Q_{i,n}\supset B_X(x_i,2^{-n})\times\{0\}$.
Let $\pip_{i,n}$ be a Lipschtiz partition of unity subordinate to the cover $Q_{i,n}$, that is,
\begin{enumerate}
\item we have $0\le \pip_{i,n}\le 1$ on $Z$,
\item the support of $\pip_{i,n}$ is contained in $Q_{i,n}$,
\item the function $\pip_{i,n}$ is $2^n\, C$-Lipschitz continuous,
\item the sum $\sum_{i,n}\pip_{i,n}=\chi_{Z_+}$.
\end{enumerate}
We set $B_{i,n}:=B_X(x_i,2^{1-n})$.
Now, for $u\in B^\theta_{2,2}(X)$ we set $Eu$ to be the function on $Z_+$ given by
\[
Eu(z):=\sum_{i,n} u_{B_{i,n}}\, \pip_{i,n}(z).
\]
If $z,w\in Q_{i_0,n_0}$, then
\begin{align*}
|Eu(z)&-Eu(w)|=\bigg\vert \sum_{i,n}[u_{B_{i,n}}-u_{B_{i_0,n_0}}][\pip_{i,n}(z)-\pip_{i,n}(w)]\bigg\vert\\
&\le C\sum_{\substack{i,n:\\Q_{i_0,n_0}\cap Q_{i,n}\ne \emptyset}}2^n d_Z(z,w)\vint_{B_{i,n}}\vint_{B_{i_0,n_0}}
|u(x_1)-u(x_2)|\, d\mu_X(x_1)\, d\mu_X(x_2)\\
&\le C\, 2^{n_0}d_Z(z,w)\, \vint_{2B_{i_0,n_0}}\vint_{2B_{i_0,n_0}}|u(x_1)-u(x_2)|\, d\mu_X(x_1)\, d\mu_X(x_2). 
\end{align*}
In the above, we used the bounded overlap property of $2B_{i_0,n_0}$.
Using the fact that if $z\in Q_{i_0,n_0}$, then $\delta(z)\simeq 2^{-n_0}$, we obtain
\begin{align*}
|Eu(z)-Eu(w)|
&\lesssim \frac{d_Z(z,w)}{\delta(z)} \vint_{2B_{i_0,n_0}}\vint_{2B_{i_0,n_0}}|u(x_1)-u(x_2)|\, d\mu_X(x_1)\, d\mu_X(x_2).
\end{align*}
Therefore, when $z\in Q_{i_0,n_0}$, we have
\[
\text{Lip}Eu(z)\lesssim \delta(z)^{-1} \vint_{2B_{i_0,n_0}}\vint_{2B_{i_0,n_0}}|u(x_1)-u(x_2)|\, d\mu_X(x_1)\, d\mu_X(x_2).
\]
For integers $n\in\Z$, we set $A_n:=\{z\in Z_+\, :\, 2^{-(n+1)}\le \delta(z)\le 2^{1-n}\}$. Then
we have $\delta(z)\simeq 2^{-n}$ when $z\in A_n$, and so
\begin{align*}
&\int_{A_n}\text{Lip}Eu(z)^2\, d\mu_a(z)
 = \int_X\int_{2^{-(n+1)}}^{2^{1-n}}\delta((x,y))^a \text{Lip}Eu((x,y))^2\, dy\, d\mu_X(x)\\
\simeq &\sum_i\int_{B_{i,n}}\int_{2^{-(n+1)}}^{2^{1-n}}\delta((x,y))^a \text{Lip}Eu((x,y))^2\, dy\, d\mu_X(x)\\
\lesssim & \sum_i\int_{B_{i,n}}\int_{2^{-(n+1)}}^{2^{1-n}}
 \vint_{2B_{i,n}}\vint_{2B_{i,n}}\frac{|u(x_1)-u(x_2)|^2}{2^{n(a-2)}}\, d\mu_X(x_1)\, d\mu_X(x_2)\, dy\, d\mu_X(x)\\
\lesssim &\sum_i\int_{B_{i,n}}\int_{2^{-(n+1)}}^{2^{1-n}}2^{n}
 \vint_{2B_{i,n}}\vint_{2B_{i,n}}\frac{|u(x_1)-u(x_2)|^2}{2^{-n(1-a)}}\, d\mu_X(x_1)\, d\mu_X(x_2)\, dy\, d\mu_X(x)\\
\simeq &\sum_i\int_{B_{i,n}} \vint_{2B_{i,n}}\vint_{2B_{i,n}}\frac{|u(x_1)-u(x_2)|^2}{2^{-n(1-a)}}\, d\mu_X(x_1)\, d\mu_X(x_2)\, d\mu_X(x)\\
\simeq&\sum_i\int_{2B_{i,n}}\vint_{2B_{i,n}}\frac{|u(x_1)-u(x_2)|^2}{2^{-n(1-a)}}\, d\mu_X(x_1)\, d\mu_X(x_2).
\end{align*}
The bounded overlap property of the balls $2B_{i,n}$ for a fixed $n\in\Z$ gives us
\[
\int_{A_n}\text{Lip}Eu(z)^2\, d\mu_a(z)
\lesssim \int_X\vint_{B(x_2,2^{1-n})}\frac{|u(x_1)-u(x_2)|^2}{2^{-n(1-a)}}\, d\mu_X(x_1)\, d\mu_X(x_2).
\]
It follows that
\begin{align*}
\int_{Z_+}\text{Lip}Eu(z)^2\, d\mu_a(z)
&\lesssim \sum_{n\in\Z}\int_X\vint_{B(x_2,2^{1-n})}\frac{|u(x_1)-u(x_2)|^2}{2^{-n(1-a)}}\, d\mu_X(x_1)\, d\mu_X(x_2)\\
&\simeq \Vert u\Vert_{B^\theta_{2,2}(X)}^2,
\end{align*}
where the last comparison is from~\cite{GKS}. 
From the above we also get the reverse  estimate of~\eqref{eq:Bes-Dir-bound}:
\begin{equation}\label{eq:Dir-Bes-bound}
\int_{Z_+}g_{Eu}^2\, d\mu_a\le C \Vert u\Vert_{B^\theta_{2,2}(X)}^2.
\end{equation}
Recall that $\theta=\tfrac{1-a}{2}$.
Finally, it remains to show that $TEu=u$ $\mu_X$-a.e.~in $X$. As from the discussion at the beginning of the proof, 
we know that for $\mu_X$-a.e.~$x\in X$
we have that $\lim_{y\to 0^+} Eu(x,y)=v(x)$ exists and is the trace $TEu$ of $Eu$ given by the condition that
\[
\lim_{r\to 0^+}\vint_{B((x,0),r)\cap Z_+}|Eu-v(x)|\, d\mu_a=0.
\]
Moreover, by the fact that $u\in L^1_{loc}(X)$, we can also ensure that $\mu_X$-a.e.~$x\in X$ is a $\mu_X$-Lebesgue 
point of $u$. Let $x\in X$ be both a $\mu_X$-Lebesgue point as well as satisfying the above condition on $v(x)$. The goal here is 
to show that $v(x)=u(x)$. To this end, let $y>0$; then
\[
Eu((x,y))-u(x)=\sum_{i,n}[u_{B_{i,n}}-u(x)]\, \phi_{i,n}(y).
\]
Observe that $\phi_{i,n}(y)\ne 0$ only when the radius of $B_{i,n}$ is comparable to $y$ and its center is at a comparable distance
from $(x,y)$ as well. Therefore
\begin{align*}
|Eu((x,y))-u(x)|&\le C\, \sum_{i,n}\left(\vint_{B((x,0),Cy)} |u-u(x)|\, d\mu_X\right)\, \phi_{i,n}(y)\\
  &= C\, \vint_{B((x,0),Cy)} |u-u(x)|\, d\mu_X\to 0\text{ as }y\to 0^+.
\end{align*}
It follows that $v(x)=u(x)$, completing the proof.
\end{proof}

The following is a type of gluing lemma that allows us to combine a Newton-Sobolev function on $Z_+$
with a Newton-Sobolev function on $Z_-:=X\times(-\infty,0)$ to obtain a Newton-Sobolev function on $Z=X\times\R$.

\begin{lemma}\label{lem:pasting}
Let $T_+$ be the trace operator on $D^{1,2}(Z_+)$ and $T_-$ be the trace operator on $D^{1,2}(Z_-)$. 
Suppose $u\in D^{1,2}(Z_+)$ and $v\in D^{1,2}(Z_-)$ such that  $T_+u(x)=T_-v(x)$ for $\mu_X$-a.e.~$x\in X$. Define $w:Z\to\R$ by 
\[
w(x,y)=\begin{cases} u(x,y)&\text{ if }y>0,\\
   T_+u(x) &\text{ if }y=0,\\
   v(x,y)&\text{ if }y<0. \end{cases}
\] 
Then $w$ belongs to $D^{1,2}(Z)$. 
\end{lemma}

\begin{proof}
Let $f=T_+u$. Then from Proposition~\ref{prop-Trace} we know that $f\in B^\theta_{2,2}(X)$; we set 
$w_0$ on $Z$ by 
\[
w_0(x,y)=\begin{cases}Ef(x,y)&\text{ if }y>0,\\
   f(x)&\text{ if }y=0,\\
   Ef(x,-y)&\text{ if }y<0. \end{cases}
\]
Then from the properties of $Ef$ obtained in the proof of Proposition~\ref{prop-Trace} and from Lemma~\ref{lem:tensor-N1p}
we know that $w_0\in D^{1,2}(Z)$. By the local nature of minimal $2$-weak upper gradients,  it follows that with $g$ the minimal
$2$-weak upper gradient of $Ef$, we have that $g_0(x,y)=g(x,y)$ when $y>0$ and $g_0(x,y)=g(x,-y)$ when $y<0$ is the
minimal $2$-weak upper gradient of $w_0$ in $Z$.  

Now, consider $\pip:=w-w_0$. Then $T_+\pip=0=T_-\pip$. Note that $\pip\vert_{Z_+}\in D^{1,2}(Z_+)$ and 
$\pip\vert_{Z_-}\in D^{1,2}(Z_-)$. Hence from~\cite{BjSh}, $\pip\vert_{\overline{Z_+}}\in D^{1,2}(\overline{Z_+})$
and $\pip\vert_{\overline{Z_-}}\in D^{1,2}(\overline{Z_-})$.  We now show that the zero extension of 
$\pip\vert_{\overline{Z_+}}$ to $Z$ is in $D^{1,2}(Z)$. Indeed, if $g\in L^2(Z_+)$ is an upper gradient of 
$\pip\vert_{\overline{Z_+}}$ and $\gamma$ is a non-constant compact rectifiable curve in $Z$ with end points denoted
$x$ and $y$, then if both $x,y\in\overline{Z_-}$, we have that 
$\pip\vert_{\overline{Z_+}}(y)=0=\pip\vert_{\overline{Z_+}}(x)$ and so the upper gradient inequality is satisfied by
the pair $\pip\vert_{\overline{Z_+}}, g$ with respect to $\gamma$. Hence, without loss of generality, we may 
assume that $x\in Z_+$.
If $\gamma$ lies entirely in $\overline{Z_+}$, then again the above pair satisfies the upper gradient inequality 
with respect to $\gamma$. If $\gamma$ intersects both $Z_+$ and $Z_-$, then 
by breaking $\gamma$ into two subcurves if necessary, we may assume that $y\in \overline{Z_-}$. Now
let $\beta$ be the largest subcurve of $\gamma$ beginning at $x$ and
lying entirely in $Z_+$, and let $z\in X\times\{0\}$ be the end-point of $\beta$. 
Then the closure of the path $\beta$ has $z$ as its
terminal point, and this compact path, also denoted $\beta$, lies entirely in $\overline{Z_+}$. Therefore 
the pair $\pip\vert_{\overline{Z_+}}, g$ satisfies the upper gradient inequality with respect to $\beta$. Since
$\pip\vert_{\overline{Z_+}}(z)=\pip_+(z)=0=\pip_+(y)$, with $g_+$ the zero-extension of $g$ to $Z$, it also follows that
$\pip_+, g_+$ satisfy the upper gradient inequality on the subcurve of $\gamma$ left-over from $\beta$. 
Hence $\pip_+, g_+$ satisfy the upper gradient inequality on $\gamma$ itself. It follows from the arbitrariness of
$\gamma$ that $g_+\in L^2(Z)$ is an upper gradient of $\pip_+$, that is, $\pip_+\in D^{1,2}(Z)$. 
A similar argument gives that $\pip_- \in D^{1,2}(Z)$. Therefore $\pip=\pip_++\pip_-$ lies in $D^{1,2}(Z)$. Finally it follows
that $w=\pip+w_0\in D^{1,2}(Z)$, completing the proof.
\end{proof}

\section{Existence and uniqueness of Cheeger-Harmonic Extensions}\label{Sect:Exist-CheegerHarm}

 Given a bounded domain $\Omega\subset X$, let us denote 
 \[
 U_{\Omega}=Z_+\cup Z_-\cup (\Omega\times\{0\}).
 \] 
 Given a function $f$ on $X$, we want to consider the Cheeger harmonic solutions to the 
 Dirichlet problem in $U_{\Omega}$ with boundary values $f$ on 
 $\partial U_{\Omega}= (X\setminus\Omega)\times \{0\}$.  That is,  we want to find a 
 function $u\in D^{1,2}(U_{\Omega})$ such that 
\begin{equation}
\label{eq:DP}
\Delta_a u=0\text{ on }U_{\Omega}\ \text{ with }\ Tu=f\text{ on } (X\setminus\Omega)\times\{0\}.
\end{equation}
Here, $\Delta_a$ is the infinitesimal generator on $Z$ associated with the Cheeger differential structure 
constructed in Subsection~\ref{Subsect:diff-struct-tensor} above. Observe that the first condition above is equivalent to knowing
that whenever $h\in D^{1,2}(U_{\Omega})$ with compact support in the domain 
$U_{\Omega}$, then as $\mu_a(Z\setminus U_\Omega)=0$ and by Lemma~\ref{lem:pasting}, we have
\begin{equation}\label{eqn:weaksolution}
\Ead(u,h)=\int_{U_{\Omega}}\langle D_a u(x,y),D_ah(x,y)\rangle_{(x,y)}\, d\mu_a(x,y)=0.
\end{equation}
Here $D_a$ is the tensor product of the Cheeger differential form $D_X$ on $X$ and the Euclidean differential form on $\R$,
see Remark~\ref{rm:Da-deff} above.

Observe that a function $u$ that satisfies~\eqref{eq:DP} as above for all compactly supported $h\in D^{1,2}(U_\Omega)$ also 
is a minimizer of the Cheeger energy in the following sense: for each $v\in D^{1,2}(Z)$ such that $v=u$ outside a compact
subset of $U_\Omega$, we have
\begin{equation}\label{eqn:minproblem1}
\Ead(u,u)\le \Ead(v,v).
\end{equation}

\begin{remark}\label{rem:Cheeger-harm-qmin}
Combining the construction of the Cheeger differential structure $D_a$ from Subsection~\ref{Subsect:diff-struct-tensor}
and Lemma~\ref{lem:tensor-innerprod-uppergrad} tell us that 
a Cheeger harmonic function $u$  in $U$ is a quasiminimizer in the sense of~\cite{KiSh}, and so is necessarily locally
H\"older continuous on $U$, and if it is non-negative, then satisfies a Harnack inequality also.
\end{remark}

\begin{thm}\label{thm:cheegerminimizer} 
Suppose that $\Omega\subset X$ is a bounded domain with $\mu_X(X\setminus\Omega)>0$. 
Then, there exists a linear extension operator $H: B^\theta_{2,2}(X)\rightarrow D^{1,2}(Z)$ 
such that  $u_{f}=Hf$ is the unique Cheeger harmonic function in $U_{\Omega}$ with $Tu=f$ on $\partial U_\Omega$ for every 
$f\in B^\theta_{2,2}(X)$. Moreover $\|Hf\|_{\infty}\leq \|f\|_{\infty}$ whenever $f$ is bounded.
\end{thm}

\begin{proof}
Fix $f\in B^{\theta}_{2,2}(X)$. Let us consider the minimization problem
\begin{equation}\label{eqn:minproblem}
E_{min}=\inf_{v\in\mathcal A} \Ead(v,v),
\end{equation}
where the class of admissible functions $\mathcal A$ consists of functions $u\in D^{1,2}(Z)$ such that $Tu=f$ on 
$\partial U_{\Omega}$. We know that the infimum above is finite as we can extend symmetrically the extension 
$Ef$ given by  
Proposition~\ref{prop-Trace} to $Z_{-}$, see also Lemma~\ref{lem:pasting}. Let $\{u_{k}\}_{k}$ be a minimizing sequence 
for~\eqref{eqn:minproblem}. Let us start by considering the case $|f|\leq 1$ at every point. In this case, 
we can assume without loss of generality that also $|u_{k}|\leq 1$ everywhere as truncation to interval $[-1,1]$ 
can only decrease the energy $\mathcal E_{a}(u_{k},u_{k})$.

Next we prove that 
\[
\Ead(u_{k}-u_{l},u_{k}-u_{l})\rightarrow 0 \quad\textrm{as }\quad k,l\rightarrow \infty.
\]
Let $u_{k,l}=\tfrac12 (u_{k}+ u_{l})$. Notice that $u_{k,l}$ satisfies also the boundary condition 
$Tu_{k,l}=f$ as  $Tu_{k}=Tu_{l}=f$. By the triangle inequality, 
\[\Ead (u_{k,l}, u_{k,l})\leq \tfrac12 \left(\Ead (u_{k}, u_{k}) +\Ead (u_{l}, u_{l}) \right).
\]
 Thus we can conclude that $\Ead(u_{k,l},u_{k,l})\rightarrow E_{min}$ as $k,l\rightarrow\infty$.
Now it follows by the parallelogram identity that
\[
\lim_{k,l\rightarrow\infty}\Ead(u_{k}-u_{l},u_{k}-u_{l})
=\lim_{k,l\rightarrow\infty} \left(2 \Ead(u_{k},u_{k}) +2 \Ead(u_{l},u_{l})-4\mathcal E_{a}(u_{k,l},u_{k,l})\right)=0.
\]

By assumption, $\mu_X(X\setminus\Omega)>0$. Hence for sufficiently large balls $B\subset Z$ centered at a point
in $\Omega\times\{0\}$ we have $\mu_X(B\cap(X\setminus\Omega))>0$. If $X$ is unbounded, we can in addition 
consider a ball $B\subset Z$ centered at $X\times \{0\}$ and with large enough radius so that 
$\mu_{X}(B\cap \Omega\times \{0\})\leq  \tfrac12\mu_{X}(B\cap X\times\{0\})$.
Either way, we see that for sufficiently large balls $B$ we have from Lemma~\ref{lem:cylinder-modulus}
that $\text{Cap}_2(B\cap\{u_k-u_l=0\})\ge \widehat{C}_B$ for some $\widehat{C}_B>0$ that is independent of $k,l$.  
Therefore by~\eqref{eq:Mazya}, we have 
\[
\int_{B}|(u_{k}-u_{l})|^2\mu_{a}
  \leq C_B\, \int_{B}|D_{a}u_{k}-D_{a}u_{l}|^2d\mu_{a}\rightarrow 0\quad\textrm{as }k,l\rightarrow\infty.
\]
As the above inequality holds for all large enough balls, this implies that there exists a limiting function $u_{f}\in D^{1,2}(Z)$ such that 
$u_{k}\rightarrow u_{f}$ in $L^2_{loc}(Z)$. This limit is obtained as follows. For each ball $B\subset Z$ centered at a
point in $(x,0)\in Z$ with a fixed choice of $x\in \Omega$, we have $u_k\to u_f$ in $N^{1,2}(B)$, and as 
$u_k-Ef\in N^{1,2}_0(\overline{B}\cap (Z_{\pm}\cup\Omega\times\{0\}))$, it follows from the Banach space property of $N^{1,2}(\overline{B}\cap Z_{\pm})$
that $u_f-Ef\in N^{1,2}(\overline{B}\cap Z_{\pm})$. As this happens for all balls $B$, it follows that pointwise $\mu_X$-a.e.~we
have $Tu_f=f$ as given by~\eqref{eq:Trace}.

As $\{\Ead(u_k,u_k)\}$ converges to the minimum and $Tu_{k}=f$ on $\partial U_{\Omega}$, the function $u_{f}$ is the 
unique global minimizer of \eqref{eqn:minproblem}. 
To see the uniqueness of the minimizer, suppose that $v\in D^{1,2}(Z)$ is also a solution. Then, as 
$u=v=f$ on $\partial U_\Omega$, it follows that $\Ead(u_f,u_f)=\Ead(v,v)$. Therefore,
\[
\Ead(u_f-v,u_f-v)=2\Ead(u_f,u_f)-2\Ead(u_f,v)=2\Ead(u_f,u_f-v)=0,
\]
the latter equality following from the Euler-Lagrange formulation that equate minimization of the energy
$\Ead$ with Definition~\ref{def:CheegerHarm}. It follows from the Poincar\'e inequality
that $u_f-v$ is a constant on $Z$. Since $Tu_f=Tv=f$ on $X\setminus\Omega$ and the $2$-capacity of 
$X\setminus\Omega$ is positive (see Lemma~\ref{lem:cylinder-modulus}), 
it follows that $u_f=v$, that is, the solution is unique.

We set $Hf=u_{f}$. As a minimizer of \eqref{eqn:minproblem}, 
$u_{f}$ is a Cheeger harmonic function in the sense that it satisfies \eqref{eqn:weaksolution} in $U_{\Omega}$. 
For unbounded and nonnegative $f$, we define the extension operator as
\[
Hf=\sum_{k=0}^\infty Hf_{k},
\]
where $f_{k}=\max\{\min\{f-k,1\},0\}$. Observe that
\[
\Eadp(Hf,Hf)=\sum_{k=0}^\infty \Eadp(Hf_k,Hf_k).
\]
Then for general $f\in B^{2,2}_{\theta}(X)$, we set 
$Ef=Ef^+-Ef^-$, where $f^+$ and $f^-$ are the positive and negative parts of $f$. As linear combinations of 
Cheeger harmonic functions are also Cheeger harmonic, this construction leads to Cheeger harmonic 
extensions of $f$, which are the unique global minimizers of the energy \eqref{eqn:minproblem}.
Note that
\[
\Eadp(Hf,Hf)=\Eadp(Hf^+,Hf^+)+\Eadp(Hf^-,Hf^-).
\]
\end{proof}

\begin{prop}\label{prop:symmetrysolution}
The extension given by Theorem~\ref{thm:cheegerminimizer} is symmetric, namely we have $u(x,y)=u(x,-y)$ for every $(x,y)\in Z$.
\end{prop}
\begin{proof}
Let $u$ be a minimizer of \eqref{eqn:minproblem}.  Let us define
\[
\widetilde u(x,y)=\frac12\left(u(x,y)+u(x,-y)\right).
\]
As $u(x,-y)$ has the same trace as $u(x,y)$ on $X\setminus\Omega$, we have $T\widetilde u=f(x)$ for almost 
every $x\in X\setminus\Omega$ and thus $\widetilde u$ satisfies the trace condition for the minimization problem. 
By the inner product structure of the energy, we see that 
\[
\mathcal E_{a}(\widetilde u, \widetilde u)
\leq \tfrac12(\mathcal E_{a}( u(x,y), u(x,y))+ \mathcal E_{a}( u(x,-y), u(x,-y)))= \mathcal E_{a}( u, u)
\] 
and the equality holds only if $D_{a}u= D_{a}\widetilde u$ $\mu_{a}$-almost everywhere, which implies that 
$\widetilde u=u$ and thus $u$ is symmetric.
\end{proof}

\section{Cheeger Harmonic Solution in \texorpdfstring{$U_{\Omega} \subset Z$}{Uomega}}
\label{Sect:Ex-Sol-DP}

\begin{thm}\label{thm:cheegerharm2} 
Suppose $f \in B_{2,2}^{\theta}(X) \cap L^2(X)$ such that $f$ is a solution to the Dirichlet problem $(-\Delta_X)^{\theta} f=0$ on $\Omega$.  Define the function 
$\tilde{u}$ by $\tilde{u}(x,0)=f(x)$ and $\tilde{u}(x,y)=\Pi_af(x,|y|)$ when $y\ne 0$, where $\Pi_af$ is defined 
in~\eqref{eq:explicit}. Then $\tilde{u}$ is Cheeger harmonic in $U_{\Omega}$. 
\end{thm}

\begin{proof}
Since $u\in D^{1,2}(Z_+)$ by Proposition~\ref{prop:trace-energy} 
and as $\tilde{u}$ is obtained by reflection with the same trace, then $\tilde{u}\in N^{1,2}_{loc}(Z)$.
Fix a compactly supported Lipschitz function $h$ in $D^{1,2}(U_{\Omega}) $.
Let $\eps>0$ small enough such that $\{|y|< \eps \} \cap \text{supp}(h)\subset \subset \Omega \times \R$. By 
Lemma~\ref{lem:tensor-innerprod-uppergrad} and integration by parts, we have
\begin{align*}
&\int_{U_{\Omega}}\langle D_a u(x,y),D_ah(x,y)\rangle_{(x,y)}\, d\mu_a
= \int_{U_{\Omega}} \langle D_X \tilde{u}, D_X h\rangle + \langle \partial_y \tilde{u}, \partial_y h\rangle   d\mu_a\\
 &\int_{U_{\Omega}} -\langle \Delta_X \tilde{u},  h\rangle + \langle \partial_y \tilde{u}, \partial_y h\rangle  d\mu_a 
 =\int_{U_{\Omega} \cap \{|y| \ge \eps \} } -\langle \Delta_X \tilde{u},  h\rangle  
    + \langle \partial_y \tilde{u}, \partial_y h\rangle  d\mu_a \\
 &+ \int_{U_{\Omega} \cap \{|y| < \eps \} } -\langle \Delta_X \tilde{u},  h\rangle  
    + \langle \partial_y \tilde{u}, \partial_y h\rangle   d\mu_a=I+J.
\end{align*}
A straightforward computation shows that 
\begin{align*}
I&=\int_{U_{\Omega} \cap \{|y| \ge \eps \} } -\langle \Delta_X \tilde{u},  h\rangle 
  + \langle \partial_y \tilde{u}, \partial_y h\rangle |y|^a   d\mu_a\\
&=\int_{U_{\Omega} \cap \{|y| \ge \eps \} } -\langle \Delta_X \tilde{u},  h\rangle  - \langle (\dfrac{a}{y}\partial_y 
   + \partial_{yy}) \tilde{u}, h\rangle   d\mu_a + \int_{\partial(U_{\Omega} \cap \{|y| \ge \eps \})  } |y|^a \partial_y u \, h \, d\mu_a\\
&=2\int_{(\Omega \times \R) \cap \{y = \eps \}} |y|^a \partial_y u \, h \, d\mu_X.
\end{align*}
Thanks to \cite[Lemma 3.1]{BLS} we have 
\[
\lim_{y\to 0^+}- \dfrac{2^{2\theta -1} \Gamma(\theta)}{\Gamma(1-\theta)}\int_X y^a\partial_y u \phi \,d\mu_X = \mathcal{E}_\theta(f,\phi),
\]
that vanishes by  assumption, for each $\phi \in L^2(X) \cap B^\theta_{2,2}(X)$. The claim then follows for Lipschitz $h$ since dominated convergence gives $h(\cdot,y) \to h(\cdot, 0)$ in $L^2(X) \cap B^\theta_{2,2}(X)$ as $y\to 0$, and since the argument in \cite[Lemma 3.1]{BLS} shows that $y^a\partial_y u $ is uniformly bounded in the dual of  $L^2 \cap B^{\theta}_{2,2}$.  Therefore $\lim_{\eps \to 0}I=0$.
Moreover, by the proof of~\cite[Lemma 3.1]{BLS}, $\Delta_X \tilde{u}$ and $\tilde{u}_y$ are locally integrable in $Z \setminus \{t = 0\}$, then  we have  $\lim_{\eps \to 0} J=0$.
\end{proof}

\section{Proof of the Main Theorems}\label{sec:proof_main_theorems}

We now collect tools from earlier in the paper to prove the main theorems. 

\begin{proof}[Proof of Theorem \ref{thm:minimization}]  
Let $f\in B^\theta_{2,2}(X)$ such that $\mathcal{E}_\theta(f,f)$ is finite,
and set $\mathcal{K}_f$ to be the collection of all functions $h\in B^\theta_{2,2}(X)$ such that
$h=f$ $\mu_X$-a.e.~in $X\setminus\Omega$. Then $\mathcal{K}_f$ is a convex subset of $B^\theta_{2,2}(X)$. Moreover,
if $(f_k)_k$ is a sequence from $\mathcal{K}_f$ such that $f_k\to f_\infty$ in $L^2(\Omega)$, then
the extension of $f_\infty$ by $f$ to $X\setminus\Omega$, also denoted $f_\infty$, satisfies
$\mathcal{E}_\theta(f,f)\le \liminf_k\mathcal{E}_\theta(f_k,f_k)$. It follows that $f_\infty\in\mathcal{K}_f$ provided the limit inferior is finite.
Let 
\[
I:=\inf\{\mathcal{E}_\theta(h,h)\, :\, h\in\mathcal{K}_f\}.
\]
If $I=0$, then necessarily $f$ is constant on $X\setminus\Omega$, and then the extension of $f$ to $\Omega$ by that constant
yields the desired solution. Hence without loss of generality we may assume that $I>0$.

As $I\le \mathcal{E}_\theta(f,f)$, it follows that $I$ is finite. Let $h_k\in\mathcal{K}_f$ be a corresponding minimizing sequence.
Without loss of generality we may assume that $\mathcal{E}_\theta(h_k,h_k)\le 2I$. It follows that
$\mathcal{E}_\theta(h_k-f,h_k-f)\le 6I+\mathcal{E}_\theta(f,f)<\infty$.
Then as $h_k-f=0$ on $X\setminus\Omega$ and so belongs to $L^2(X)$ as well,
by Proposition~\ref{prop:besovenergy} 
we know that the sequence $(h_k-f)$ is bounded in the Besov seminorm as well;
$\Vert h_k-f\Vert_{B^\theta_{2,2}(X)}^2\le C(6I+\mathcal{E}_\theta(f,f))$. As $f\in B^\theta_{2,2}(X)$, it follows now that
$\Vert h_k\Vert_{B^\theta_{2,2}(X)}^2\le C(6I+\mathcal{E}_\theta(f,f))+\Vert f\Vert_{B^\theta_{2,2}(X)}^2$.

For each positive integer $k$ consider the function $v_k$ on $X\times X$ given by
$v_k(x,w)=h_k(x)-h_k(w)$. Then, equipping $X\times X$ by the measure $\nu$ given by
\[
d\nu(x,w)=\frac{1}{d(x,w)^{2\theta}\mu_X(B(x,d_X(x,w)))}\, d\mu_X(x)\, d\mu_X(w),
\]
we see that $v_k\in L^2(X\times X,\nu)$ is a bounded sequence. Hence there is a function $v_\infty\in L^2(X\times X,\nu)$
such that  (a convex combination of) $v_k\to v_\infty$ in $L^2(X\times X,\nu)$. By passing to a subsequence if necessary,
we may also assume that this convergence is $\nu$-almost everywhere as well (and hence, $\mu_X\times\mu_X$-almost
everywhere as well). Note that if both $x,w\in X\setminus\Omega$, then $v_\infty(x,w)=f(x)-f(w)$. If $x\in\Omega$
and $w\in X\setminus\Omega$ such that $v_\infty(x,w)=\lim_kv_k(x,w)$, then
$v_\infty(x,w)=\lim_kh_k(x)-f(w)$, and so we set $h_\infty$ to be the function on $X$ given by
$h_\infty(x)=f(x)$ when $x\in X\setminus\Omega$ and $h_\infty(x)=v_\infty(x,w)+f(w)$ with $w\in X\setminus\Omega$ as chosen 
above. With the aid of Fubini's theorem, we know that for $\mu_X$-almost every $x\in\Omega$ we can find 
$w\in X\setminus\Omega$
such that $v_\infty(x,w)=\lim_kv_k(x,w)$. Thus $h_\infty$ is well-defined. Moreover, 
the function $X\times X\ni (x,w)\mapsto h_\infty(x)-h_\infty(w)$ is the function $v_\infty$, and so 
$h_\infty\in B^\theta_{2,2}(X)$, with $h_\infty=f$ $\mu_X$-a.e.~in $X\setminus\Omega$; that is, $h_\infty\in\mathcal{K}_f$.
The above argument also shows that $h_k\to h_\infty$ in $L^2_{loc}(X)$ (and indeed, $h_k-h_\infty\to 0$ in 
$L^2(X)$ as $\Omega$ is bounded).
Finally, by the lower semicontinuity and the bilinearity of $\mathcal{E}_\theta$, we see that
\[
 I\le \mathcal{E}_\theta(h_\infty,h_\infty)\le \liminf_k\mathcal{E}_\theta(h_k,h_k)=I,
 \]
 and so $h_\infty$ is the desired solution.
\end{proof}

\begin{proof}[Proof of Theorem \ref{thm:mainthm1}] 
Let $f$ be as in the statement. Then $T\Pi_a f=f$ on $X$ where $\Pi_af$ is as in~\eqref{eq:explicit}. 
We will denote the reflection of $\Pi_af$ along $X$ by setting
$\Pi_af(x,-y):=\Pi_af(x,y)$ for $y>0$. Then, by Theorem~\ref{thm:cheegerharm2}, we know that $\Pi_af$ is Cheeger harmonic
in $U_\Omega$ and therefore is a quasiminimizer in the sense of~\cite{KiSh}. Hence $\Pi_af$ is locally H\"older continuous
on $U_\Omega$ and hence necessarily also on $\Omega\times\{0\}$. As $T\Pi_af=f$, it follows that $f$ is 
locally H\"older continuous on $\Omega$.
\end{proof}

Now we are ready to prove the final main theorem of the paper, Theorem~\ref{thm:main3}.

\begin{proof}[Proof of Theorem~\ref{thm:main3}]
We know from Theorem~\ref{thm:minimization} that given $f\in B^\theta_{2,2}(X)\cap L^2(X)$ that there is a function 
$f_0\in B^\theta_{2,2}(X)\cap L^2(X)$ such that $f_0=f$ in $X\setminus\Omega$ and $f_0$ is a solution to
the problem~\eqref{eq:fract-Laplace}. So let $f$ simply denote such a solution. 

Consider $Hf$, the solution to the Dirichlet problem on $U_\Omega$ with boundary data $f$, as 
constructed in Theorem~\ref{thm:cheegerminimizer}, and the function $\Pi_af$ on $Z$ as in
Theorem~\ref{thm:cheegerharm2} above. Then both $Hf$ and $\Pi_af$ are Cheeger $2$-harmonic in $U_\Omega$
with trace $f$ on $\partial U_\Omega=(X\setminus\Omega)\times\{0\}$, and so by the uniqueness of such
solution as shown in the proof of Theorem~\ref{thm:cheegerminimizer}, we obtain that $\Pi_af=Hf$, and so
$Hf$ is the unique solution to the Dirichlet problem studied. 

The maximum principle follows immediately if $\text{esssup}_{w\in X\setminus\Omega}f(w)=\infty$, and so we may assume
without loss of generality that this supremum is finite. Setting 
\[
M=\text{esssup}_{w\in X\setminus\Omega}f(w),
\]
the Markov property of $\mathcal{E}_\theta$ implies that 
$\mathcal{E}_\theta(\min\{f,M\},\min\{f,M\})\le \mathcal{E}_\theta(f,f)$ with $\min\{f,M\}=f$ 
$\mu_X$-a.e.~on $X\setminus\Omega$;
thus $\min\{f,M\}$ is a solution. By the uniqueness of the solution, it follows that $\min\{f,M\}=f$ on $X$, thus proving
the maximum principle. The strong maximum principle follows from the analogous principle for $Hf$ in $U_\Omega$
upon noting that $(X\setminus\Omega)\times\{0\}=\partial U_\Omega$, see~\cite{KiSh}.
\end{proof}

The three main theorems together demonstrate the existence, uniqueness, and regularity of the solution to 
the Dirichlet problem related to the fractional Laplacian on $X$.

\vskip .5cm

\noindent Author Information:

\vskip .3cm

\noindent Sylvester Eriksson-Bique

\noindent Address: Department of Mathematics, Jyv\"askyl\"a University,  P.O. Box 35 (MaD), FI-40014, Jyv\"askyl\"a, Finland

\noindent Email: {\tt syerikss@jyu.fi}

\vskip .3cm

\noindent Gianmarco Giovannardi

\noindent Address: Departamento de Geometría y Topología,
Universidad de Granada, E-18071 Granada, Spain.

\noindent Email: {\tt giovannardi@ugr.es}

\vskip .3cm

\noindent Riikka Korte

\noindent Address: Department of Mathematics and Systems Analysis, Aalto University, P.O.~Box~11100, FI-00076 Aalto, Finland.

\noindent Email: {\tt riikka.korte@aalto.fi}

\vskip .3cm

\noindent Nageswari Shanmugalingam

\noindent Address: Department of Mathematical Sciences, University of Cincinnati, \\P.O.Box~210025, Cincinnati, OH 45221-0025, USA.

\noindent Email: {\tt shanmun@uc.edu}

\vskip .3cm

\noindent Gareth Speight

\noindent Address: Department of Mathematical Sciences, University of Cincinnati,\\ P.O.Box~210025, Cincinnati, OH 45221-0025, USA.

\noindent Email: {\tt speighgh@ucmail.uc.edu}

\end{document}